\documentclass[12pt]{amsart}
\usepackage[sc]{mathpazo}

\usepackage[top=3cm,bottom=3cm,left=3cm,right=3cm]{geometry}

\usepackage{amsmath, amsfonts, amssymb, tikz-cd, mathtools, amsthm}
\usepackage[all]{xy}
\usepackage{indentfirst}
\usepackage{enumitem}
\usepackage{graphicx}
\usepackage{float}
\usepackage{mathrsfs}
\usepackage[normalem]{ulem}

\newtheorem{proposition}{Proposition}
\newtheorem{Proposition}[proposition]{Proposition}
\newtheorem{theorem}[proposition]{Theorem}
\newtheorem{lemma}[proposition]{Lemma}
\newtheorem{corollary}[proposition]{Corollary}

\newtheorem{Theorem}[proposition]{Theorem}
\newtheorem{Lemma}[proposition]{Lemma}
\newtheorem{Corollary}[proposition]{Corollary}

\theoremstyle{definition}
\newtheorem{definition}[proposition]{Definition}

\newtheorem{Definition}[proposition]{Definition}

\theoremstyle{remark}
\newtheorem{example}[proposition]{Example}

\newtheorem{Example}[proposition]{Example}

\usepackage{prettyref}

\newrefformat{Theorem}{Theorem \ref{#1}}
\newrefformat{Definition}{Definition \ref{#1}}
\newrefformat{Lemma}{Lemma \ref{#1}}
\newrefformat{Corollary}{Corollary \ref{#1}}
\newrefformat{fig}{Figure \ref{#1}}
\newrefformat{sec}{Section \ref{#1}}
\newrefformat{sub}{Subsection \ref{#1}}
\newrefformat{subsec}{Subsection \ref{#1}}

\usepackage{caption}

\usepackage{hyperref}
\definecolor{upsPurp}{RGB}{99,0,60}
\definecolor{upsRed}{RGB}{198,11,70}
\definecolor{upsCyan}{RGB}{0,148,181}
\hypersetup{
     colorlinks=true,
     linkcolor=upsRed,
     citecolor = upsPurp,      
     urlcolor= upsCyan,
     }

\usepackage{cleveref}

\usepackage[
backend=biber,
style=alphabetic,
sorting=nty
]{biblatex}

\title{Oriented graphs on curve complex I: hyperbolic and extremal length}
\author{Dong Tan}
\address{Guangxi University, School of Mathematics and Information Science \& Guangxi Center for Mathematical Research \& Center
for Applied Mathematics of Guangxi, East Daxue Road, Xixiangtang District,
Nanning, Zip Code: 530004, P.R. China} 
\email{duzuizhe2013@foxmail.com}
\author{Wen Yang}
\address{Hunan University, School of Mathematics, Lushan Road (S), Yuelu District,
Changsha, Zip Code: 410082, P.R. China}
\email{yang-wen@139.com}
\thanks{The first author is supported by
National Natural Science Foundation of China (No: 12001122, 12271533, 12361014)
and the special foundation for Guangxi Ba Gui Scholars.
The second author is the corresponding author and is partially supported by
Fundamental Research Funds for the Central Universities (No: 531118010617)
and National Natural Science Foundation of China (No: 12471073).}



\date{}

\begin{document}

\begin{abstract}
    We investigate oriented graphs based on the curve complex $C(S)$ of a closed surface $S$ and induced by functions on the vertex set of $C(S)$. In particular, we introduce the Dehn quasi-homothetic functions, which behave similarly to homotheties under repeated Dehn twists. We prove that any two positive such functions of the same type induce different oriented graphs unless they are proportional. This leads to a new rigidity result for closed hyperbolic surfaces---distinct from the $9g-9$ theorem and length spectrum rigidity---knowing only for any two disjoint simple closed curves which one is longer (in terms of hyperbolic or extremal length) suffices to determine the hyperbolic metric on the surface. We also prove that each automorphism of the oriented graph induced by a function with sublevel sets finite is induced by a self-homeomorphism of $S$.
\end{abstract}

\maketitle

\section{Introduction}\label{sec:Introduction}

A goal of isospectral geometry is to recover the geometry of a Riemannian manifold from its Laplacian spectrum. This idea is metaphorically phrased by Mark Kac \cite{kac1966can} as: ``Can one hear the shape of a drum?'' For closed hyperbolic surfaces, the answer is affirmative: since the Laplacian spectrum and the length spectrum reciprocally determine each other \cite[Satz 7 and Satz 8]{huber1959analytischen}, and the length spectrum suffices to distinguish almost all closed hyperbolic surfaces \cite{wolpert1977eigenvalue}, it follows that the Laplacian spectrum also suffices to distinguish them. In fact, we don't need the full length spectrum; only a part of it, namely the simple length spectrum, is enough for this distinction \cite{baik2020simple}.

The simple length spectrum is commonly defined as the data formed by arranging the lengths of all simple closed curves on the surface in increasing order while recording their multiplicities. This process is akin to a musician extracting all the notes from a musical score and arranging them in increasing order, yet surprisingly, one can almost always reconstruct the original score from these essentially shuffled notes.

Nevertheless, most people lack the perceptual precision to perceive the simple length spectrum, thus failing to hear the shape of hyperbolic surfaces in this way. While most people can only discern relative differences between adjacent notes and thus reconstruct the score roughly, are such reconstructions sufficient to distinguish different music pieces? This inspires us to ask the analogous question in the context of curve complex. Instead of recording the lengths of curves on the curve complex, we now record the order relations between the lengths of adjacent vertices (i.e., disjoint simple closed curves) in the complex, forming what we call an oriented graph of the curve complex. 
We find that any hyperbolic metric on a closed surface is uniquely determined by the corresponding oriented graph on the curve complex. This is akin to reconstructing an exact musical score from a rough version obtained through relative pitch, which indicates it is easier to hear the shape of a closed hyperbolic surface.

In fact, we can say even more. After all, the ``musical score'' written on the curve complex need not originate from a hyperbolic metric on the surface. Let us begin with a closed surface $S$, and let $C(S)$ be its curve complex. Denote the vertex set of $C(S)$ by $V(S)$. Then as illustrated in \prettyref{fig:induce-oriented-graph}, every real-valued function $L:V(S)\rightarrow\mathbb{R}$ induces an oriented graph $G_{L}$, whose vertex set is $V(S)$, and for any $\alpha,\beta\in V(S)$, there is a directed edge from $\alpha$ to $\beta$ in $G_{L}$ if and only if $\alpha$ and $\beta$ are disjoint and $L(\alpha)<L(\beta)$. Here we regard $L$ as a ``musical score'' on the curve complex, while $G_{L}$ is its rough version. Obviously, for a generic $L$, one cannot recover $L$ from $G_{L}$. This is not surprising, since we only expect to reconstruct music---not all sounds---from a rough version of its score. We can reconstruct the music because the music itself has an inherent structure.

\begin{figure}[H]
	\noindent \begin{centering}
		\includegraphics{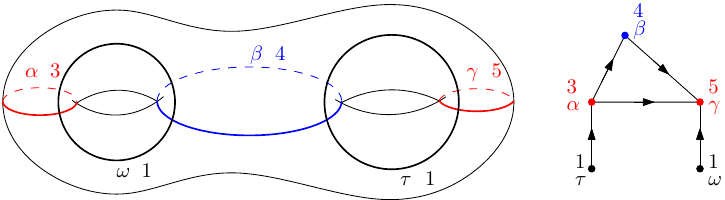}
		\par\end{centering}
	\caption{\label{fig:induce-oriented-graph}Here $L$ takes the values $3,1,4,1,5$ at the vertices $\alpha,\omega,\beta,\tau,\gamma$,
respectively.}
\end{figure}

Let us consider a class of functions that can be imagined as certain types of musical scores, with hyperbolic length functions as a special case:
\begin{Definition}\label{Definition:Dehn quasi-homothetic}
We will denote by $E^{c}(S)$ the set $\{(\alpha,\beta)\in V(S)\times V(S):$
there is no edge joining $\alpha$ to $\beta$ in $C(S)\}$. Let
$L$ be a real-valued function on $V(S)$. We say that $L$ is a \emph{Dehn
quasi-homothetic} function of type $(f,A)$, if there exist two functions
$f:\mathbb{N}\rightarrow\mathbb{R}_{+}$ and $A:E^{c}(S)\rightarrow\mathbb{R}_{+}$
satisfying the following three conditions:

(1) $f(n)\rightarrow+\infty$ as $n\rightarrow+\infty$;

(2) For any positive integer $N$, the set
\[
\{f(k)/f(j):j,k>N\}
\]
is dense in $\mathbb{R}_{+}$;

(3) For any $(\alpha,\beta)\in E^{c}(S)$ we have
\[
\lim_{n\rightarrow+\infty}\frac{L(T_{\alpha}^{n}\beta)}{f(n)}=A(\alpha,\beta)\cdot L(\alpha),
\]
where $T_\alpha$ is the Dehn twist along $\alpha$.
\end{Definition}

The first main result of this paper is that we can almost reconstruct
a Dehn quasi-homothetic and positive function $L$ from $G_{L}$ if
we know its type:
\begin{Theorem}
\label{Theorem:main1}Let $L_{1},L_{2}$ be non-negative real-valued
functions on $V(S)$. Suppose that they are Dehn quasi-homothetic
functions of the same type and induce the same oriented graph. Then for any pair of adjacent vertices $\alpha$ and $\beta$ in $C(S)$, it holds that
\[
L_{1}(\alpha)\cdot L_{2}(\beta)=L_{2}(\alpha)\cdot L_{1}(\beta).
\]
Furthermore, if $L_{2}$ is positive, the above equation is equivalent to 
\[
\frac{L_{1}(\alpha)}{L_{2}(\alpha)}=\frac{L_{1}(\beta)}{L_{2}(\beta)},
\]
and hence there exists a constant $k\geqslant0$ such that $L_{1}=k\cdot L_{2}$.
\end{Theorem}

Since the hyperbolic length functions and the square roots of the
extremal length functions of Riemann surfaces
are Dehn quasi-homothetic functions of the same type, we have:
\begin{Corollary}\label{Corollary:Cor3}
Let $\mathcal{T}(S)$ be the Teichm\"{u}ller space of $S$. For any
$X,Y\in\mathcal{T}(S)$, let $C_{X}(S)$ denote the oriented graph
induced by the hyperbolic or extremal length function of $X$, and
$C_{Y}(S)$ similarly for $Y$. Then $C_{X}(S)=C_{Y}(S)$ implies
$X=Y$.
\end{Corollary}

There is an alternative interpretation of this corollary. Let $D(S)$ denote the space of all oriented graphs on $C(S)$ induced by real-valued functions on $V(S)$, and let $\Phi:\mathbb{R}^{V(S)}\rightarrow D(S)$ be the map that sends each such function to the oriented graph induced by it. Consider
$$
\mathcal{T}(S)\xrightarrow[\mathrm{Ext}]{\mathrm{Hyp}}\mathbb{R}^{V(S)}\xrightarrow{\Phi}D(S),
$$
where $\mathrm{Hyp}:\mathcal{T}(S)\rightarrow\mathbb{R}^{V(S)}$ and $\mathrm{Ext}:\mathcal{T}(S)\rightarrow\mathbb{R}^{V(S)}$ send each $X\in\mathcal{T}(S)$ to its hyperbolic and extremal length function, respectively. It is well known that both $\mathrm{Hyp}$ and $\mathrm{Ext}$ are injective, and their images in $\mathbb{R}^{V(S)}$ are disjoint. Our corollary asserts that, although $\Phi$ forgets most of the information, the mappings $\Phi\circ\mathrm{Hyp}$ and $\Phi\circ\mathrm{Ext}$ remain injective, and their images in $D(S)$ remain disjoint as well.

Furthermore, we can think of $C_{X}(S)$ as a discrete version of the hyperbolic surface $X$. This suggests a connection to Ivanov's famous metaconjecture:
``Every object naturally associated to a surface $S$ and having a sufficiently rich structure has $\mathrm{Mod}_{S}^{\pm}$ as its automorphism group. Moreover, this can be proved by a reduction to the theorem about the automorphisms of $C(S)$.'' So a natural question arises: does $C_{X}(S)$ have $\mathrm{Mod}_{S}^{\pm}$ as its automorphism group?

The answer is negative. In fact, since $C_{X}(S)$ is a discrete version of $X$, its automorphism group is roughly the isometry group of $X$. Nevertheless, every automorphism $\varphi$ of $C_{X}(S)$ is induced by a self-homeomorphism $f$ of $S$, i.e., $f(\alpha)=\varphi\cdot\alpha$ for all $\alpha\in V(S)$.

Let us consider a more general case. Let $L$ be a real-valued function on $V(S)$, and let $G_{L}$ be the oriented graph induced by $L$. Is every automorphism of $G_{L}$ induced by a self-homeomorphism of $S$? Without additional constraints on $L$, the answer is negative, since there is no substantial relationship between $G_{L}$ and $S$.

Nevertheless, we find that imposing a single condition on $L$ is enough, though it is seemingly unrelated to $S$:
\begin{Definition}
We say that a function $L:V(S)\rightarrow\mathbb{R}$ has \emph{finite sublevel sets} if $L^{-1}((-\infty,r])$ is finite for every $r\in\mathbb{R}$.
\end{Definition}

Our second main result is:
\begin{Theorem}\label{Theorem:main2}
If $L$ has finite sublevel sets, then every automorphism of $G_{L}$ is induced by a self-homeomorphism of $S$.
\end{Theorem}

Since hyperbolic and extremal length functions of Riemann surfaces have finite sublevel sets, we can prove that:
\begin{Corollary}\label{Corollary:Cor6}
Let $\mathcal{T}(S)$ be the Teichmüller space of $S$. For any $X\in\mathcal{T}(S)$, let $C_{X}(S)$ denote the oriented graph induced by the hyperbolic or extremal length function of $X$. Then every automorphism of $C_{X}(S)$ is induced by an isometry of $X$.
\end{Corollary}

\subsection*{Some remarks on history.}

To the author's knowledge, McShane and Parlier were the first to study
the relative lengths of all pairs of simple closed geodesics on closed
surfaces or surfaces with specified boundary lengths \cite{mcshane2008multiplicities}.
They showed that the length ordering of simple closed geodesics uniquely
determines a point in Teichm\"{u}ller space. In \cite{parlier2025orderingcurvessurfaces},
Parlier, Vo, and Xu extended this work by considering the length orderings
of closed geodesics on hyperbolic surfaces, including those with self-intersections,
thereby removing the need for fixed boundary length. However, we focus
on a different perspective: instead of studying hyperbolic length
functions on closed curves on surfaces of finite type, we consider
the Dehn quasi-homothetic functions on simple closed curves on closed
surfaces. Our shift from metric to topology is based on the concept
of oriented graph, which provides a framework for understanding the
underlying combinatorial structure.

\subsection*{Outline of the paper.}

In \prettyref{sub:examples} we present some examples of Dehn quasi-homothetic
functions to explain further the setting we adopt. In particular,
we find that the hyperbolic and extremal length functions have the
same type. We devote \prettyref{subsec:Dehn quasi-homothetic} to
the rigidity of Dehn quasi-homothetic functions of the same type on
infinite graphs rather than the curve complex $C(S)$, and \prettyref{Theorem:main1}
follows. We then prove \prettyref{Theorem:main2} in \prettyref{subsec:finitely-small}.
Finally, we prove \prettyref{Corollary:Cor3} and \prettyref{Corollary:Cor6}
in \prettyref{sec:applications}.

\subsection*{Acknowledgments.}

The authors would like to express their sincere gratitude to Qiyu Chen, Xiaoming Du, Lixin Liu, Yusen Long, Sicheng Lu, Huiping Pan, Yaozhong Shi, Weixu Su, Hao Sun, Bin Xu and Youliang Zhong for their helpful conversations and encouragement.

\section{Preliminaries}\label{sec:pre}

Throughout this paper, let $S$ be an orientable closed surface with
genus $g\geq2$.

\subsection{Geometric intersection number}

A \emph{simple closed curve} is a topological embedding $\gamma:S^{1}\rightarrow S$,
where $S^{1}$ is the standard circle. For simplicity, we shall refer
to simple closed curves as \emph{curves}, since no other types of
curves will be discussed. Note that we will identify a curve with
its image in $S$. We say a curve is \emph{essential }if it is not
null-homotopic.

For any simple closed curves $\alpha_{1},\alpha_{2}$ on
$S$, the \emph{geometric intersection number} between $\alpha_{1}$
and $\alpha_{2}$ is defined as 
\[
i(\alpha_{1},\alpha_{2})\coloneqq\inf\{\#(\alpha_{1}'\cap\alpha_{2}'):\alpha_{j}'\text{ is freely isotopic to }\alpha_{j}\}.
\]
In other words, $i(\alpha_{1},\alpha_{2})$ is the minimal number
of intersection points between $\alpha_{1}$ and $\alpha_{2}$, up
to free isotopy.

Note that we will identify each simple closed curve
$\alpha$ with its free isotopy class $[\alpha]$. For any free isotopy
classes $[\alpha_{1}]$ and $[\alpha_{2}]$, the geometric intersection
number between them is well defined by setting $i([\alpha_{1}],[\alpha_{2}])\coloneqq i(\alpha_{1},\alpha_{2})$.
Furthermore, $[\alpha_{1}]$ and $[\alpha_{2}]$ have disjoint representatives
if and only if $i([\alpha_{1}],[\alpha_{2}])=0$. Therefore we say that $[\alpha_{1}]$ and $[\alpha_{2}]$ are \emph{disjoint}
if $[\alpha_{1}]\neq[\alpha_{2}]$ and $i([\alpha_{1}],[\alpha_{2}])=0$.

We say that $\alpha_{1},\cdots,\alpha_{3g-3}\in V(S)$ form a \emph{system
of decomposing curves} if for any $j,k$ we have $i(\alpha_{j},\alpha_{k})=0$.

For more details we refer to Chapter 1 of \cite{farb2012primer}.

\subsection{Curve complex}

The curve complex $C(S)$ was introduced as a simplicial complex by
Harvey \cite{harvey1981boundary}. However, throughout this paper,
we only consider the 1-skeleton of $C(S)$, which is an undirected
graph defined as follows: the vertex set $V(S)$ consists of isotopy
classes of essential curves in $S$, and two vertices are connected
by an edge if and only if they are disjoint. For brevity,
we will identify the curve complex with its 1-skeleton.
Note that in \cite{harvey1981boundary} Harvey showed that $C(S)$ is a connected graph.

We say that a bijection $\varphi:V(S)\rightarrow V(S)$ is an \emph{automorphism}
of $C(S)$ if for any $\alpha,\beta\in V(S)$ we have
\[
i(\alpha,\beta)=0\Leftrightarrow i(\varphi\cdot\alpha,\varphi\cdot\beta)=0.
\]
It is clear that every self-homeomorphism $f:S\rightarrow S$ naturally
induces an automorphism $\varphi$ of $C(S)$ by setting $\varphi\cdot\alpha\coloneqq f(\alpha)$.
On the other hand, let $\mathrm{Mod}^{\pm}(S)$ denote the extended
mapping class group consisting of all self-homeomorphisms of $S$, 
it is well known that:

\begin{theorem}[Ivanov, \cite{ivanov1997automorphisms}]\label{Theorem:Ivanov}
Let $S$ be a closed surface of genus $g \geqslant 2$. Then every automorphism of the curve complex $C(S)$ is induced by an element of the extended mapping class group $\operatorname{Mod}^{ \pm}(S)$.
\end{theorem}

For more details on curve complex we refer to Section 4.1 of \cite{farb2012primer}.

\subsection{Oriented graph on curve complex}\label{sec:oriented_curve_graph}

Although in \prettyref{sec:Introduction} we have briefly introduced the oriented graphs induced by functions, we now give a more precise definition in this subsection.

\begin{definition}
Recall that $V(S)$ is the vertex set of the curve complex $C(S)$.
An \emph{oriented graph} on $C(S)$ is an asymmetric binary relation
$G\subseteq V(S)\times V(S)$ such that
\[
(\alpha,\beta)\in G\Rightarrow i(\alpha,\beta)=0,
\]
where $i(\alpha,\beta)$ is the geometric intersection number between
$\alpha$ and $\beta$. In other words, regarding $G$ as a directed
graph, $G$ has vertex set $V(S)$, contains at most one directed
edge between $\alpha$ and $\beta$ for any $\alpha,\beta\in V(S)$, and
such an edge exists only when $\alpha$ and $\beta$ are disjoint.

We say that a bijection $\varphi:V(S)\rightarrow V(S)$ is an \emph{automorphism}
of $G$ if for any $\alpha,\beta\in V(S)$ we have
\[
(\alpha,\beta)\in G\Leftrightarrow(\varphi\cdot\alpha,\varphi\cdot\beta)\in G.
\]
The \emph{automorphism group} of $G$, denoted by $\mathrm{Aut}(G)$,
consists of all automorphisms of $G$.
\end{definition}

\begin{example}
	Let $\alpha,\beta,\gamma\in V(S)$ be pairwise disjoint
	curves, and set $G=\{(\alpha,\beta),(\beta,\gamma),(\gamma,\alpha)\}$,
	then $G$ is an oriented graph on $C(S)$. Note that each $\tau\in V(S)-\{\alpha,\beta,\gamma\}$
	is still considered a vertex in $G$, even though it is not incident
	to any edge. Thus $\mathrm{Aut}(G)\cong\mathbb{Z}_{3}\times\mathrm{Sym}(\mathbb{N})$,
	where $\mathrm{Sym}(\mathbb{N})$ is the symmetric group consisting
	of all bijections from $\mathbb{N}$ to $\mathbb{N}$.
\end{example}

\begin{definition}
	Let $L: V(S) \rightarrow \mathbb{R}$ be a real-valued function on $V(S)$. The \emph{oriented graph $G_L$ induced by} $L$ is defined to be
	$$
	\{(\alpha, \beta) \in V(S) \times V(S) \mid i(\alpha, \beta)=0 \text { and } L(\alpha)<L(\beta)\}.
	$$
	This construction ensures that edges in $G_L$ encode both the disjointness of curves and the ordering imposed by $L$.
\end{definition}

When there is no ambiguity, we denote $(\alpha,\beta)\in G_{L}$ by
$\alpha<\beta$. Since $G_{L}$ is a directed graph, we say that a
sequence 
\[
\alpha_{1}<\alpha_{2}<\cdots<\alpha_{m}
\]
forms a\emph{ path} of \emph{length} $m-1$ from $\alpha_{1}$ to
$\alpha_{m}$ in $G_{L}$. Note that every $\varphi\in\mathrm{Aut}(G_{L})$
sends paths to paths and preserves their lengths.

The \emph{distance} $d_{G_{L}}(\alpha,\beta)$ from $\alpha$ to $\beta$
is defined as the length of a shortest path from $\alpha$ to $\beta$
in $G_{L}$; if no such path exists, we set $d_{G_{L}}(\alpha,\beta)\coloneqq\infty$.

We say that a subgraph $(V_{1},E_{1})$ of $G_{L}$ is a \emph{tournament}
if there is exactly one edge in $E_{1}$ between every pair of distinct
vertices in $V_{1}$.

For more details on graph theory we refer to \cite{diestel2017graph}.

\subsection{Dehn twist}

For any $\alpha\in V(S)$, the \emph{Dehn twist} along $\alpha$ is
the mapping $T_{\alpha}:V(S)\rightarrow V(S)$ that sends each $\beta\in V(S)$
to the simple closed curve $T_{\alpha}(\beta)$ obtained by the procedure
illustrated in \prettyref{fig:Dehn-twist}: near each intersection point
with $\alpha$, the curve $\beta$ is twisted to the left, following
a full loop around $\alpha$, and then continues along its original
path. If $i(\alpha,\beta)=0$, we have $T_{\alpha}(\beta)=\beta$.

\begin{figure}[H]
	\noindent \begin{centering}
		\includegraphics{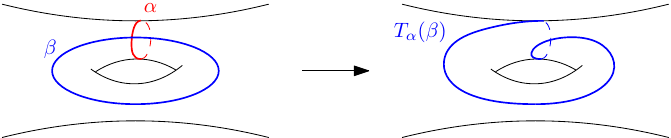}
		\par\end{centering}
	\caption{\label{fig:Dehn-twist}}
\end{figure}

Note that $T_{\alpha}$ is an automorphism of $C(S)$ and is induced
by a self-homeomorphism of $S$.

\prettyref{fig:Dehn-twist2} illustrates the case where $i(\alpha,\beta)=2$,
with the \textcolor{red}{red} and \textcolor{blue}{blue} curves representing
two segments of $\beta$ in a neighborhood of $\alpha$.

\begin{figure}[H]
	\noindent \begin{centering}
		\includegraphics{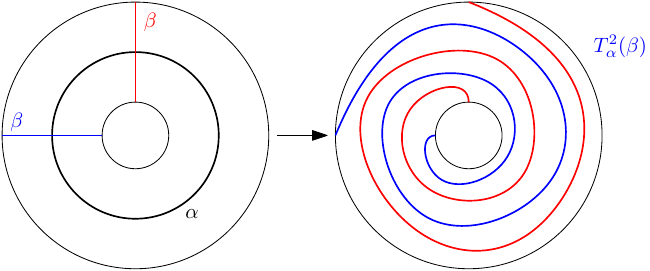}
		\par\end{centering}
	\caption{\label{fig:Dehn-twist2}}
\end{figure}

For more details we refer to Chapter 3 of \cite{farb2012primer}.

\subsection{Teichm\"{u}ller space, hyperbolic length and extremal length}

The Teichm\"{u}ller space $\mathcal{T}(S)$ consists of all isotopy
classes of hyperbolic metrics on $S$.

More specifically, by a Riemannian metric on $S$ we mean a differential
structure together with a Riemannian metric. We say that a Riemannian
metric on $S$ is a \emph{hyperbolic metric} if it has constant curvature
$-1$. For any hyperbolic metrics $\rho_{1},\rho_{2}$ on $S$, we
say that $\rho_{1}$ and $\rho_{2}$ are \emph{isotopic} (written
$\rho_{1}\sim\rho_{2}$) if there exists an isometry $\varphi:(S,\rho_{1})\rightarrow(S,\rho_{2})$
isotopic to the identity. Let $\mathcal{H}(S)$ be the set of hyperbolic
metrics on $S$. We call $\mathcal{H}(S)/\sim$ the \emph{Teichm\"{u}ller
space} of $S$ and denote it by $\mathcal{T}(S)$.

For more details on the Teichm\"{u}ller space we refer to \cite{imayoshi1992introduction}.

Let $\rho$ be a Riemannian metric on $S$, possibly degenerate at
a finite number of points. Let us denote by $A_{\rho}$ the area of
$(S,\rho)$. For any path $\alpha$ on $S$, we write $\ell_{\rho}(\alpha)$
for the integral of the arc length element of $\rho$ along $\alpha$.
For any $[\alpha]\in V(S)$, we call
\[
\inf\{\ell_{\rho}(\alpha'):\alpha'\in[\alpha]\}
\]
the $\rho$-\emph{length} of $[\alpha]$ and denote it by $L_{\rho}([\alpha])$.

Suppose that $X=[\rho]\in\mathcal{T}(S)$ and $[\alpha]\in V(S)$.

The \emph{hyperbolic length} of $[\alpha]$ on $X$, denoted by $L_{X}([\alpha])$,
is defined to be $L_{\rho}([\alpha])$. This definition does not depend
on the choice of $\rho$, since $[\alpha]$ is an isotopy class of
curves. We call $L_{X}$ the \emph{hyperbolic length function} of
$X$.

\begin{Theorem}[$9g-9$ theorem]
There exist $9g-9$ vertices $\alpha_{1},\cdots,\alpha_{9g-9}\in V(S)$
such that the mapping
\[
\Lambda:\mathcal{T}(S)\rightarrow\mathbb{R}_{+}^{9g-9},\quad X\mapsto\left(L_{X}(\alpha_{1}),\cdots,L_{X}(\alpha_{9g-9})\right)
\]
is injective.
\end{Theorem}
For more details about the $9g-9$ theorem we refer to Proposition
7.11 of \cite{fathi2012thurston}.

For any $X,Y\in\mathcal{T}(S)$, the \emph{Thurston distance} from
$X$ to $Y$ is defined to be
\begin{equation}
d_{\mathrm{Th}}(X,Y)\coloneqq\ln\sup_{\alpha\in V(S)}\frac{L_{Y}(\alpha)}{L_{X}(\alpha)}.\label{eq:Thurston-distance}
\end{equation}
In \cite{thurston1998minimal}, Thurston showed that $d_{\mathrm{Th}}$
is an asymmetric metric.

For any smooth and positive function $\lambda$ on $S$, the metric
$\lambda\rho$ defined by $(\lambda\rho)(x)=\lambda(x)\cdot\rho(x)$
is conformally equivalent to $\rho$. The \emph{extremal length} of
$[\alpha]$ on $X$, denoted by $\mathrm{Ext}_{X}([\alpha])$, is
defined to be
\[
\sup_{\lambda}\frac{L_{\lambda\rho}^{2}([\alpha])}{A_{\lambda\rho}},
\]
where the supremum is over all smooth and positive functions $\lambda$
on $S$. This definition does not depend on the choice of $\rho$,
since $[\alpha]$ is an isotopy class of curves. We call $\mathrm{Ext}_{X}$
the \emph{extremal length function} of $X$.

In \cite[Theorem 4]{kerckhoff1980asymptotic} Kerckhoff showed:
\begin{Theorem}\label{Theorem:Teich-distance}
The Teichm\"{u}ller distance between $X,Y\in\mathcal{T}(S)$ is equal
to
\[
\frac{1}{2}\ln\sup_{\alpha\in V(S)}\frac{\mathrm{Ext}_{Y}(\alpha)}{\mathrm{Ext}_{X}(\alpha)}.
\]
\end{Theorem}

It is well known that the hyperbolic length embedding $\mathrm{Hyp}:\mathcal{T}(S)\rightarrow\mathbb{R}^{V(S)}$
and the extremal length embedding $\mathrm{Ext}:\mathcal{T}(S)\rightarrow\mathbb{R}^{V(S)}$
have disjoint images in $\mathbb{R}^{V(S)}$. Similarly, the embeddings
$\mathrm{Hyp}$ and $\sqrt{\mathrm{Ext}}$ have disjoint images
even after composing with the projection
from $\mathbb{R}^{V(S)}$ to $P(\mathbb{R}^{V(S)})$:
\begin{Lemma}
\label{Lemma:hyp-neq-ext}Let $X,Y\in\mathcal{T}(S)$. Then $L_{X}\neq k\cdot\sqrt{\mathrm{Ext}_{Y}}$
for any $k>0$.
\end{Lemma}
\begin{proof}
For the convenience of the reader, we provide a proof here.

It suffices to prove a more general result: if $C$ is a geodesic
current on $S$, and $L:V(S)\rightarrow\mathbb{R}$ is defined by
$L(\cdot)=i(C,\cdot)$, where $i$ is the geometric intersection number,
then $L\neq\sqrt{\mathrm{Ext}_{Y}}$.

Suppose on the contrary that there is a geodesic current $C$ such
that $L=\sqrt{\mathrm{Ext}_{Y}}$. Since $\sqrt{\mathrm{Ext}_{Y}}$
is uniquely extended to a continuous and homogeneous function on $\mathcal{MF}$,
we have $L(F)=\sqrt{\mathrm{Ext}_{Y}(F)}$ for every $F\in\mathcal{MF}$.

Let $\alpha,\beta\in V(S)$ such that $i(\alpha,\beta)=0$. We set
$F=\alpha+\beta\in\mathcal{MF}$. Because the geometric intersection
number is bilinear, we have 
\[
L(\alpha+\beta)=i(C,\alpha+\beta)=i(C,\alpha)+i(C,\beta)=L(\alpha)+L(\beta).
\]
Hence
\[
\sqrt{\mathrm{Ext}_{Y}(F)}=L(F)=L(\alpha)+L(\beta)=\sqrt{\mathrm{Ext}_{Y}(\alpha)}+\sqrt{\mathrm{Ext}_{Y}(\beta)}.
\]
According to Theorem 3.12 and Proposition 3.13 of \cite{kahn2022conformal}, there is a
Jenkins-Strebel differential $q_{F}$ such that
\[
\mathrm{Ext}_{Y}(F)=A_{|q_{F}|}=\frac{\left(L_{|q_{F}|}(\alpha)+L_{|q_{F}|}(\beta)\right)^{2}}{A_{|q_{F}|}}.
\]
where $|q_{F}|$ is the flat metric induced by $q_{F}$. Thus
\begin{equation}
\sqrt{\mathrm{Ext}_{Y}(\alpha)}+\sqrt{\mathrm{Ext}_{Y}(\beta)}=\sqrt{\mathrm{Ext}_{Y}(F)}=\frac{L_{|q_{F}|}(\alpha)+L_{|q_{F}|}(\beta)}{\sqrt{A_{|q_{F}|}}}.\label{eq:hyp-neq-ext-1}
\end{equation}

According to the definition of extremal length, we have
\[
\sqrt{\mathrm{Ext}_{Y}(\alpha)}\geqslant\frac{L_{|q_{F}|}(\alpha)}{\sqrt{A_{|q_{F}|}}},\quad\sqrt{\mathrm{Ext}_{Y}(\beta)}\geqslant\frac{L_{|q_{F}|}(\beta)}{\sqrt{A_{|q_{F}|}}}.
\]
Combining these with \prettyref{eq:hyp-neq-ext-1} we get
\[
\sqrt{\mathrm{Ext}_{Y}(\alpha)}=\frac{L_{|q_{F}|}(\alpha)}{\sqrt{A_{|q_{F}|}}},\quad\sqrt{\mathrm{Ext}_{Y}(\beta)}=\frac{L_{|q_{F}|}(\beta)}{\sqrt{A_{|q_{F}|}}}.
\]
Thus $|q_{F}|$ is the extremal metric for both $\alpha$ and $\beta$
on $Y$.

However, the extremal metric for $\alpha$ on $Y$ is $|q_{\alpha}|$
and is unique up to scaling (see Theorem 2.2 of \cite{jenkins2012univalent}), where $q_{\alpha}$ is the Jenkins-Strebel
differential corresponding to $\alpha$ in Theorem 3.12 of \cite{kahn2022conformal}. From
the construction of $q_{F}$ and $q_{\alpha}$, it follows that 
$|q_{F}|\neq k\cdot|q_{\alpha}|$ for any $k$. Hence $|q_{F}|$
is not the extremal metric for $\alpha$ on $Y$, a contradiction.
\end{proof}

\subsection{Examples of Dehn quasi-homothetic function}\label{sub:examples}

\subsubsection{Metric of non-positive curvature}

In \cite[Lemma 2.2]{slegers2022energy} Slegers actually proved:
\begin{Lemma}
Let $\rho$ be a Riemannian metric of non-positive curvature on $S$,
possibly degenerate at a finite number of points. Then for any $\alpha,\beta\in V(S)$,
there exists a constant $C=C(\rho,\alpha,\beta)>0$, such that for
any $k\in\mathbb{N}$ we have
\[
\left|L_{\rho}(T_{\alpha}^{k}\beta)-k\cdot i(\alpha,\beta)\cdot L_{\rho}(\alpha)\right|\leqslant C.
\]
\end{Lemma}
It follows that the $\rho$-length function $L_{\rho}:V(S)\rightarrow\mathbb{R}$
is a Dehn quasi-homothetic function of type 
\[
(k\mapsto k,(\alpha,\beta)\mapsto i(\alpha,\beta))
\]
 if $\rho$ is a hyperbolic or flat metric on $S$.

\subsubsection{Extremal length}

The closure of the image of $\mathbb{R}_{+}\times V(S)$ under
\[
i_{*}:\mathbb{R}_{+}\times V(S)\rightarrow\mathbb{R}^{V(S)},\quad(t,\alpha)\mapsto t\cdot i(\alpha,\cdot)
\]
is called the space of \emph{measured foliations} on $S$ and denoted
by $\mathcal{MF}$.

In \cite[Proposition 3]{kerckhoff1980asymptotic} Kerckhoff proved:
\begin{Proposition}
For any $X\in\mathcal{T}(S)$, there is a unique continuous extension
of the extremal length function $\mathrm{Ext}_{X}$ from $V(S)$ to
$\mathcal{MF}$ such that $\mathrm{Ext}_{X}(r\cdot F)=r^{2}\cdot\mathrm{Ext}_{X}(F)$
holds for any $r\in\mathbb{R}_{+}$ and $F\in\mathcal{MF}$.
\end{Proposition}
Consequently, for any $\alpha,\beta\in V(S)$ satisfying $i(\alpha,\beta)\neq0$,
it follows from
\[
\lim_{k\rightarrow+\infty}\frac{T_{\alpha}^{k}\beta}{k\cdot i(\alpha,\beta)}=\alpha
\]
that
\[
\sqrt{\mathrm{Ext}_{X}(\alpha)}=\lim_{k\rightarrow+\infty}\sqrt{\mathrm{Ext}_{X}\left(\frac{T_{\alpha}^{k}\beta}{k\cdot i(\alpha,\beta)}\right)}=\lim_{k\rightarrow+\infty}\frac{\sqrt{\mathrm{Ext}_{X}(T_{\alpha}^{k}\beta)}}{k\cdot i(\alpha,\beta)}.
\]
Therefore the square root of the extremal length function of $X\in\mathcal{T}(S)$
is a Dehn quasi-homothetic function of type 
\[
(k\mapsto k,(\alpha,\beta)\mapsto i(\alpha,\beta)).
\]

\subsubsection{Why the same type?}

The following example illustrates why we require that $L_{1}$ and
$L_{2}$ are Dehn quasi-homothetic functions of the same type:
\begin{Example}
Let $L_{1}$ be a Dehn quasi-homothetic function of type $(f,A_{1})$. We
assume further that for any $\alpha,\beta\in V(S)$ satisfying $i(\alpha,\beta)\neq0$,
there exists $c=c(\alpha,\beta)>0$ such that for any $k\in\mathbb{N}$
we have
\[
\left|L_{1}(T_{\alpha}^{k}\beta)-f(k)\cdot A_{1}(\alpha,\beta)\cdot L_{1}(\alpha)\right|\leqslant c.
\]
Note that the hyperbolic length function is an example of $L_{1}$.

Let $h:\mathbb{R}\rightarrow\mathbb{R}$ be a strictly increasing
bijection such that there exists $c_{h}>0$ satisfying $|h(x)-x|\leqslant c_{h}$
for any $x\in\mathbb{R}$. We set $L_{2}=h\circ L_{1}$. Then $L_{1}$
and $L_{2}$ induce the same oriented graph.

Let us set
\[
A_{2}:E^{c}(S)\rightarrow\mathbb{R},\quad(\alpha,\beta)\mapsto\frac{L_{1}(\alpha)}{L_{2}(\alpha)}\cdot A_{1}(\alpha,\beta).
\]
Then for any $\alpha,\beta\in V(S)$ satisfying $i(\alpha,\beta)\neq0$
and any $k\in\mathbb{N}$, we have
\begin{align*}
 & \phantom{\leqslant}\left|L_{2}(T_{\alpha}^{k}\beta)-f(k)\cdot A_{2}(\alpha,\beta)\cdot L_{2}(\alpha)\right|\\
 & \leqslant\left|(h\circ L_{1})(T_{\alpha}^{k}\beta)-L_{1}(T_{\alpha}^{k}\beta)+L_{1}(T_{\alpha}^{k}\beta)-f(k)\cdot A_{2}(\alpha,\beta)\cdot L_{2}(\alpha)\right|\\
 & \leqslant\left|(h\circ L_{1})(T_{\alpha}^{k}\beta)-L_{1}(T_{\alpha}^{k}\beta)\right|+\left|L_{1}(T_{\alpha}^{k}\beta)-f(k)\cdot\frac{L_{1}(\alpha)}{L_{2}(\alpha)}\cdot A_{1}(\alpha,\beta)\cdot L_{2}(\alpha)\right|\\
 & \leqslant c_{h}+c.
\end{align*}
It follows that $L_{2}$ is a Dehn quasi-homothetic function of type
$(f,A_{2})$, similar to the type of $L_{1}$.
\end{Example}
On the other hand, if $L$ is a Dehn quasi-homothetic function of
type $(f_{1},A)$, and $f_{2}:\mathbb{N}\rightarrow\mathbb{R}$ satisfies
\[
\lim_{n\rightarrow+\infty}\frac{f_{1}(n)}{f_{2}(n)}=1,
\]
then $L$ is also a Dehn quasi-homothetic function of type $(f_{2},A)$.
We do not know yet whether there exist
nontrivial Dehn quasi-homothetic functions $L_{i}$ of type
$(f_{i},A)$, where
\[
\lim_{n\rightarrow+\infty}\frac{f_{1}(n)}{f_{2}(n)}\neq1,
\]
such that $L_{1}$ and $L_{2}$ induce the same oriented graph.

\section{General theory}

\subsection{Dehn quasi-homothetic functions on infinite graphs}\label{subsec:Dehn quasi-homothetic}

In order to reveal the essence of \prettyref{Theorem:main1}, we will
establish a more general result on infinite graphs. Let us first fix
an infinite graph $G$, which can be viewed as a generalization of
the curve complex $C(S)$. We denote its vertex set and edge set by
$V$ and $E$ respectively. For any subset $W$ of $V$, let $G[W]$
denote the subgraph induced in $G$ by $W$.

We first propose a generalization of Dehn twists to infinite graphs.
\begin{Definition}
Suppose that $T:V\rightarrow\mathrm{Aut}(G)$ sends each $\alpha\in V$
to an automorphism of $G$. Write $T_{\alpha}\coloneqq T(\alpha)$. We
say that $T$ is a \emph{system of Dehn twists} on $G$ if for any
$\alpha,\beta\in V$,

(1) $T_{\alpha}\alpha=\alpha$,

(2) $T_{\alpha}\beta=\beta$ if there is an edge joining $\alpha$
to $\beta$,

(3) $\#\{T_{\alpha}^{k}\beta:k\in\mathbb{Z}\}=\infty$ if there is
no edge joining $\alpha$ to $\beta$.
\end{Definition}
Note that not every infinite graph has a system of Dehn twists since
this definition requires that the automorphism group is large.

In the following we assume that $G$ is equipped with a system $T$
of Dehn twists.
\begin{Lemma}\label{Lemma:Dehn-twist-rotate}
Let $\alpha,\beta,\gamma\in V$ such that $\{\alpha,\gamma\},\{\beta,\gamma\}\in E$. Then
$\{T_{\alpha}^{k}\beta,\gamma\}\in E$ for any $k\in\mathbb{Z}$.
\end{Lemma}
\begin{proof}

Since $T_{\alpha}$ is an automorphism, we get $\{T_{\alpha}\beta,T_{\alpha}\gamma\}\in E$
from $\{\beta,\gamma\}\in E$. Moreover, because $T$ is a system
of Dehn twists, we get $T_{\alpha}\gamma=\gamma$ from $\{\alpha,\gamma\}\in E$,
and hence $\{T_{\alpha}\beta,\gamma\}\in E$. Similarly, we have $\{T_{\alpha}^{-1}\beta,\gamma\}\in E$.
Repeating this process, we obtain the desired result.

\end{proof}

\prettyref{Definition:Dehn quasi-homothetic} extends naturally to
the setting of infinite graphs:
\begin{Definition}
We will denote by $E^{c}$ the set $\{(\alpha,\beta)\in V\times V:\{\alpha,\beta\}\notin E\}$. Let
$L$ be a real-valued function on $V$. We say that $L$ is a \emph{Dehn
quasi-homothetic} function of type $(f,A)$ on $(G,T)$, if there
exist two functions $f:\mathbb{N}\rightarrow\mathbb{R}_{+}$ and $A:E^{c}\rightarrow\mathbb{R}_{+}$
satisfying the following three conditions:

(1) $f(n)\rightarrow+\infty$ as $n\rightarrow+\infty$;

(2) For any positive integer $N$, the set 
\[
\{f(k)/f(j):j,k>N\}
\]
is dense in $\mathbb{R}_{+}$;

(3) For any $(\alpha,\beta)\in E^{c}$ we have 
\[
\lim_{n\rightarrow+\infty}\frac{L(T_{\alpha}^{n}\beta)}{f(n)}=A(\alpha,\beta)\cdot L(\alpha).
\]
\end{Definition}

Every function $L:V\rightarrow\mathbb{R}$ induces an oriented graph
$G_{L}\subset V\times V$, where $(\alpha,\beta)\in G_{L}$ if and
only if $\{\alpha,\beta\}\in G$ and $L(\alpha)<L(\beta)$. Recall
that we use $\alpha<\beta$ to denote $(\alpha,\beta)\in G_{L}$.

\begin{Theorem}\label{Theorem:main1.1}
Let $(G,T)$ be an infinite graph equipped with a system of Dehn twists. Assume
that for any $\{\alpha,\beta\}\in E$ there exists $\{\xi,\eta\}\in E$
such that
\[
G[\alpha,\beta,\xi,\eta]=\xymatrix{\alpha\ar@{-}[r]\ar@{-}[d] & \beta\ar@{-}[d]\\
\eta\ar@{-}[r] & \xi
}
.
\]
Let $L_{1},L_{2}$ be non-negative real-valued functions on $V$. Suppose
that they are Dehn quasi-homothetic functions of the same type and
induce the same oriented graph. Then for any $\{\alpha,\beta\}\in E$
we have
\[
L_{1}(\alpha)\cdot L_{2}(\beta)=L_{2}(\alpha)\cdot L_{1}(\beta).
\]
Furthermore, if $L_{2}>0$, the above equation is equivalent to 
\[
\frac{L_{1}(\alpha)}{L_{2}(\alpha)}=\frac{L_{1}(\beta)}{L_{2}(\beta)},
\]
and hence there exists a constant $k\geqslant0$ such that $L_{1}=k\cdot L_{2}$
if $G$ is connected.
\end{Theorem}

\begin{proof}
Our proof is inspired by the recent work of Parlier, Vo and Xu \cite[Section 4]{parlier2025orderingcurvessurfaces}
on length orders.

Suppose on the contrary that there exists $\{\alpha,\beta\}\in E$
such that
\[
L_{1}(\alpha)\cdot L_{2}(\beta)\neq L_{2}(\alpha)\cdot L_{1}(\beta).
\]
Without loss of generality we may assume that $L_{1}(\alpha)\cdot L_{2}(\beta)>L_{2}(\alpha)\cdot L_{1}(\beta)$.

Let us denote by $(f,A)$ the type of $L_{1}$ and $L_{2}$ as Dehn
quasi-homothetic functions.

By the assumption, there exist $\xi,\eta\in V$ such that
\[
G[\alpha,\beta,\xi,\eta]=\xymatrix{\alpha\ar@{-}[r]\ar@{-}[d] & \beta\ar@{-}[d]\\
\eta\ar@{-}[r] & \xi
}
.
\]

For each $k\in\mathbb{N}_{+}$, we set $\xi_{k}\coloneqq T_{\alpha}^{k}\xi$
and $\eta_{k}\coloneqq T_{\beta}^{k}\eta$. Then it follows from \prettyref{Lemma:Dehn-twist-rotate}
that for any $j,k\in\mathbb{N}_{+}$ we have
\begin{equation}
G[\alpha,\beta,\xi_{j},\eta_{k}]=\xymatrix{\alpha\ar@{-}[r]\ar@{-}[d] & \beta\ar@{-}[d]\\
\eta_{k}\ar@{-}[r] & \xi_{j}
}
.\label{eq:main1-1}
\end{equation}

For $m=1,2$, since $L_{m}$ is Dehn quasi-homothetic function of
type $(f,A)$, we have 
\begin{align*}
\lim_{j\rightarrow+\infty}\frac{L_{m}(\xi_{j})}{f(j)}=A(\alpha,\xi)\cdot L_{m}(\alpha) & ,\\
\lim_{k\rightarrow+\infty}\frac{L_{m}(\eta_{k})}{f(k)}=A(\beta,\eta)\cdot L_{m}(\beta) & .
\end{align*}
Hence for any $\varepsilon>0$ there exists $K(\varepsilon)>0$ such
that for any $j,k>K(\varepsilon)$ we have 
\begin{align*}
\left|\frac{L_{m}(\xi_{j})}{f(j)}-A(\alpha,\xi)\cdot L_{m}(\alpha)\right| & <\varepsilon,\\
\left|\frac{L_{m}(\eta_{k})}{f(k)}-A(\beta,\eta)\cdot L_{m}(\beta)\right| & <\varepsilon,
\end{align*}
and so 
\begin{equation}
\begin{cases}
L_{1}(\xi_{j})>f(j)\cdot(A(\alpha,\xi)\cdot L_{1}(\alpha)-\varepsilon),\\
L_{1}(\eta_{k})<f(k)\cdot(A(\beta,\eta)\cdot L_{1}(\beta)+\varepsilon),
\end{cases}\quad\begin{cases}
L_{2}(\xi_{j})<f(j)\cdot(A(\alpha,\xi)\cdot L_{2}(\alpha)+\varepsilon),\\
L_{2}(\eta_{k})>f(k)\cdot(A(\beta,\eta)\cdot L_{2}(\beta)-\varepsilon).
\end{cases}\label{eq:main1-2}
\end{equation}

On the other hand, since $L_{1}(\alpha)\cdot L_{2}(\beta)>L_{2}(\alpha)\cdot L_{1}(\beta)\geqslant0$,
we have 
\[
(A(\alpha,\xi)\cdot L_{1}(\alpha))\cdot(A(\beta,\eta)\cdot L_{2}(\beta))>(A(\alpha,\xi)\cdot L_{2}(\alpha))\cdot(A(\beta,\eta)\cdot L_{1}(\beta))\geqslant0,
\]
Therefore we can choose $\varepsilon>0$ such that
\[
(A(\alpha,\xi)\cdot L_{1}(\alpha)-\varepsilon)\cdot(A(\beta,\eta)\cdot L_{2}(\beta)-\varepsilon)>(A(\alpha,\xi)\cdot L_{2}(\alpha)+\varepsilon)\cdot(A(\beta,\eta)\cdot L_{1}(\beta)+\varepsilon)
\]
and the numbers
\[
A(\alpha,\xi)\cdot L_{1}(\alpha)-\varepsilon,\quad A(\beta,\eta)\cdot L_{2}(\beta)-\varepsilon,\quad A(\alpha,\xi)\cdot L_{2}(\alpha)+\varepsilon,\quad A(\beta,\eta)\cdot L_{1}(\beta)+\varepsilon
\]
are positive. Now we have
\[
\frac{A(\alpha,\xi)\cdot L_{1}(\alpha)-\varepsilon}{A(\beta,\eta)\cdot L_{1}(\beta)+\varepsilon}>\frac{A(\alpha,\xi)\cdot L_{2}(\alpha)+\varepsilon}{A(\beta,\eta)\cdot L_{2}(\beta)-\varepsilon}.
\]
According to the definition of Dehn quasi-homothetic function, there
exist $j,k>K(\varepsilon)$ such that 
\[
\frac{A(\alpha,\xi)\cdot L_{1}(\alpha)-\varepsilon}{A(\beta,\eta)\cdot L_{1}(\beta)+\varepsilon}>\frac{f(k)}{f(j)}>\frac{A(\alpha,\xi)\cdot L_{2}(\alpha)+\varepsilon}{A(\beta,\eta)\cdot L_{2}(\beta)-\varepsilon}.
\]
Thus \prettyref{eq:main1-2} implies
\[
\begin{cases}
L_{1}(\xi_{j})>L_{1}(\eta_{k}),\\
L_{2}(\xi_{j})<L_{2}(\eta_{k}).
\end{cases}
\]
Combining this with \prettyref{eq:main1-1} we obtain $\xi_{j}>\eta_{k}$
in $G_{L_{1}}$ and $\xi_{j}<\eta_{k}$ in $G_{L_{2}}$, which contradicts
the fact that $L_{1}$ and $L_{2}$ induce the same oriented graph.
\end{proof}

Unfortunately, it is the subgraph induced by non-separating curves,
rather than $C(S)$ itself, that satisfies the condition required
of $G$ in \prettyref{Theorem:main1.1}. Therefore we need the following
lemma:
\begin{Lemma}
\label{Lemma:separating-to-nonseparating}Let $W\subset V$ be invariant
under the actions of $T$, i.e., for each $\alpha\in V$ we have $T_{\alpha}(W)=W$. Assume
that:

(1) If $\{\alpha,\beta\}\in E$ satisfies $\alpha,\beta\in V-W$, then
there exist $\xi,\eta\in W$ such that
\[
G[\alpha,\beta,\xi,\eta]=\xymatrix{\alpha\ar@{-}[r]\ar@{-}[d] & \beta\ar@{-}[d]\\
\eta\ar@{-}[r] & \xi
}
;
\]

(2) If $\{\alpha,\beta\}\in E$ satisfies $\alpha\in V-W$ and $\beta\in W$,
then there exists $\xi\in W$ such that
\[
G[\alpha,\beta,\xi]=\xymatrix{\alpha\ar@{-}[r] & \beta\\
\xi\ar@{-}[ur]
}
.
\]

Let $L_{1},L_{2}:V\rightarrow\mathbb{R}$ be Dehn quasi-homothetic
functions of the same type on $(G,T)$. Suppose that there exists
$\{\alpha,\beta\}\in E$ such that 
\[
L_{1}(\alpha)\cdot L_{2}(\beta)\neq L_{2}(\alpha)\cdot L_{1}(\beta).
\]
Then there exists $\{\alpha',\beta'\}\in E$ with $\alpha',\beta'\in W$
such that
\[
L_{1}(\alpha')\cdot L_{2}(\beta')\neq L_{2}(\alpha')\cdot L_{1}(\beta').
\]
\end{Lemma}
\begin{proof}
Suppose that $L_{1}$ and $L_{2}$ are Dehn quasi-homothetic functions
of type $(f,A)$. It suffices to consider the cases where $\alpha$
or $\beta$ does not belong to $W$. These can be divided into two
cases.

\textbf{First case:} $\alpha,\beta\in V-W$.

By the assumption, there exist $\xi,\eta\in W$ such that
\[
G[\alpha,\beta,\xi,\eta]=\xymatrix{\alpha\ar@{-}[r]\ar@{-}[d] & \beta\ar@{-}[d]\\
\eta\ar@{-}[r] & \xi
}
.
\]

For each $k\in\mathbb{N}_{+}$, we set $\xi_{k}=T_{\alpha}^{k}(\xi)$
and $\eta_{k}=T_{\beta}^{k}(\eta)$. Then we have
\[
G[\alpha,\beta,\xi_{k},\eta_{k}]=\xymatrix{\alpha\ar@{-}[r]\ar@{-}[d] & \beta\ar@{-}[d]\\
\eta_{k}\ar@{-}[r] & \xi_{k}
}
.
\]
Moreover, since $W$ is invariant under the actions of $T$, we see
that $\xi_{k},\eta_{k}\in W$.

For $m=1,2$, since $L_{m}$ is Dehn quasi-homothetic function of
type $(f,A)$, we have 
\begin{align*}
\lim_{k\rightarrow+\infty}\frac{L_{m}(\xi_{k})}{f(k)}=A(\alpha,\xi)\cdot L_{m}(\alpha) & ,\\
\lim_{k\rightarrow+\infty}\frac{L_{m}(\eta_{k})}{f(k)}=A(\beta,\eta)\cdot L_{m}(\beta) & .
\end{align*}
Hence 
\begin{align*}
 & \hphantom{=}\lim_{k\rightarrow+\infty}\frac{L_{1}(\xi_{k})\cdot L_{2}(\eta_{k})-L_{2}(\xi_{k})\cdot L_{1}(\eta_{k})}{[f(k)]^{2}}\\
 & =\lim_{k\rightarrow+\infty}\left(\frac{L_{1}(\xi_{k})}{f(k)}\cdot\frac{L_{2}(\eta_{k})}{f(k)}-\frac{L_{2}(\xi_{k})}{f(k)}\cdot\frac{L_{1}(\eta_{k})}{f(k)}\right)\\
 & =A(\alpha,\xi)\cdot L_{1}(\alpha)\cdot A(\beta,\eta)\cdot L_{2}(\beta)-A(\alpha,\xi)\cdot L_{2}(\alpha)\cdot A(\beta,\eta)\cdot L_{1}(\beta)\\
 & =A(\alpha,\xi)\cdot A(\beta,\eta)\cdot(L_{1}(\alpha)\cdot L_{2}(\beta)-L_{2}(\alpha)\cdot L_{1}(\beta))\\
 & \neq0.
\end{align*}
It follows that $\alpha'=\xi_{k}$ and $\beta'=\eta_{k}$ have the
desired property for sufficiently large $k$.

\textbf{Second case:} Exactly one of $\alpha$ and $\beta$ belongs
to $W$.

Without loss of generality we may assume that $\alpha\notin W$ but
$\beta\in W$.

By the assumption, there exists $\xi\in W$ such that
\[
G[\alpha,\beta,\xi]=\xymatrix{\alpha\ar@{-}[r] & \beta\\
\xi\ar@{-}[ur]
}
.
\]

For each $k\in\mathbb{N}_{+}$ we set $\xi_{k}=T_{\alpha}^{k}\xi$.
Then we have $\xi_{k}\in W$ and
\[
G[\alpha,\beta,\xi_{k}]=\xymatrix{\alpha\ar@{-}[r] & \beta\\
\xi_{k}\ar@{-}[ur]
}
.
\]

For $m=1,2$, since $L_{m}$ is Dehn quasi-homothetic function of
type $(f,A)$, we have 
\[
\lim_{k\rightarrow+\infty}\frac{L_{m}(\xi_{k})}{f(k)}=A(\alpha,\xi)\cdot L_{m}(\alpha),
\]
hence 
\begin{align*}
 & \hphantom{=}\lim_{k\rightarrow+\infty}\frac{L_{1}(\xi_{k})\cdot L_{2}(\beta)-L_{2}(\xi_{k})\cdot L_{1}(\beta)}{f(k)}\\
 & =\lim_{k\rightarrow+\infty}\left(\frac{L_{1}(\xi_{k})}{f(k)}\cdot L_{2}(\beta)-\frac{L_{2}(\xi_{k})}{f(k)}\cdot L_{1}(\beta)\right)\\
 & =A(\alpha,\xi)\cdot(L_{1}(\alpha)\cdot L_{2}(\beta)-L_{2}(\alpha)\cdot L_{1}(\beta))\\
 & \neq0.
\end{align*}
It follows that $\alpha'=\xi_{k}$ and $\beta'=\beta$ have the desired
property for sufficiently large $k$.
\end{proof}
Now we are ready to generalize \prettyref{Theorem:main1.1} as follows:
\begin{Theorem}
\label{Theorem:main1.2}Let $(G,T)$ be an infinite graph equipped
with a system of Dehn twists. Let $W\subset V$ be invariant under
the actions of $T$. Assume that:

(1) If $\{\alpha,\beta\}\in E$ satisfies $\alpha,\beta\in V-W$ or
$\alpha,\beta\in W$, then there exist $\xi,\eta\in W$ such that
\[
G[\alpha,\beta,\xi,\eta]=\xymatrix{\alpha\ar@{-}[r]\ar@{-}[d] & \beta\ar@{-}[d]\\
\eta\ar@{-}[r] & \xi
}
;
\]

(2) If $\{\alpha,\beta\}\in E$ satisfies $\alpha\in V-W$ and $\beta\in W$,
then there exists $\xi\in W$ such that
\[
G[\alpha,\beta,\xi]=\xymatrix{\alpha\ar@{-}[r] & \beta\\
\xi\ar@{-}[ur]
}
.
\]

Let $L_{1},L_{2}$ be non-negative real-valued functions on $V$. Suppose
that they are Dehn quasi-homothetic functions of the same type and
induce the same oriented graph. Then for any $\{\alpha,\beta\}\in E$
we have
\[
L_{1}(\alpha)\cdot L_{2}(\beta)=L_{2}(\alpha)\cdot L_{1}(\beta).
\]
Furthermore, if $L_{2}>0$, the above equation is equivalent to 
\[
\frac{L_{1}(\alpha)}{L_{2}(\alpha)}=\frac{L_{1}(\beta)}{L_{2}(\beta)},
\]
and hence there exists a constant $k\geqslant0$ such that $L_{1}=k\cdot L_{2}$
if $G$ is connected.
\end{Theorem}
\begin{proof}
Assume that $L_{1}$ and $L_{2}$ are Dehn quasi-homothetic functions
of type $(f,A)$.

Suppose on the contrary that there exists $\{\alpha,\beta\}\in E$
such that
\[
L_{1}(\alpha)\cdot L_{2}(\beta)\neq L_{2}(\alpha)\cdot L_{1}(\beta).
\]
Then \prettyref{Lemma:separating-to-nonseparating} shows that there
exists $\{\alpha',\beta'\}\in E$ with $\alpha',\beta'\in W$ such
that
\[
L_{1}(\alpha')\cdot L_{2}(\beta')\neq L_{2}(\alpha')\cdot L_{1}(\beta').
\]

Since $W$ is invariant under the actions of $T$, it is natural to
induce a system of Dehn twists on $G[W]$ by $T$. Specifically, we
define $T|_{W}:W\rightarrow\mathrm{Aut}(G[W])$ by setting
\[
(T|_{W})_{\sigma}(\gamma)\coloneqq T_{\sigma}\gamma
\]
for all $\sigma,\gamma\in W$. Then $T|_{W}$ is a system of Dehn twists on $G[W]$.

Similarly, it is natural to induce Dehn quasi-homothetic functions
on $(G[W],T|_{W})$ by $L_{1}$ and $L_{2}$. Specifically, we set
$E^{c}|_{W}\coloneqq\{(\sigma,\gamma)\in W\times W:\{\sigma,\gamma\}\notin E\}$.
Then the restrictions of $L_{1}$ and $L_{2}$ to $W$ are Dehn quasi-homothetic
functions of type $(f,A|_{(E^{c}|_{W})})$ on $(G[W],T|_{W})$.

Since $(G[W],T|_{W},L_{1}|_{W},L_{2}|_{W})$ satisfies the conditions
required of $(G,T,L_{1},L_{2})$ in \prettyref{Theorem:main1.1},
we conclude that
\[
L_{1}(\alpha')\cdot L_{2}(\beta')=L_{2}(\alpha')\cdot L_{1}(\beta'),
\]
a contradiction!
\end{proof}

\begin{proof}[Proof of \prettyref{Theorem:main1}]
Setting $G=C(S)$ and choosing $W$ to be the set of non-separating
curves, we see that \prettyref{Theorem:main1.2} specializes to \prettyref{Theorem:main1},
as desired.
\end{proof}

\subsection{Functions with sublevel sets finite}\label{subsec:finitely-small}

First, we list some elementary properties of the distance $d_{G_{L}}$, especially its weak correspondence with the geometric intersection number:
\begin{lemma}\label{Lemma:distance-basic}
Let $G_{L}$ be the oriented graph induced by a function $L:V(S)\to\mathbb{R}$.
Then for any $\alpha,\beta\in V(S)$ and any $\varphi\in\mathrm{Aut}(G_{L})$
we have:

\begin{description}

\item [(1)] $d_{G_{L}}(\alpha,\beta)=d_{G_{L}}(\varphi\cdot\alpha,\varphi\cdot\beta)$;

\item [(2)] If $L(\alpha)\geqslant L(\beta)$, then $d_{G_{L}}(\alpha,\beta)=\infty$;

\item [(3.1)] If $d_{G_{L}}(\alpha,\beta)=1$, then $i(\alpha,\beta)=0$;

\item [(3.2)] If $2\leqslant d_{G_{L}}(\alpha,\beta)<\infty$, then $i(\alpha,\beta)>0$;

\item [(3.3)] If $d_{G_{L}}(\alpha,\beta)=\infty$ and $L(\alpha)<L(\beta)$,
then $i(\alpha,\beta)>0$.

\end{description}
\end{lemma}
 
\begin{proof}
\textbf{(1)} This follows from the fact that both $\varphi$ and $\varphi^{-1}$ send
paths to paths and preserve their lengths.

\textbf{(2)} The existence of a path from $\alpha$ to $\beta$ would
imply $L(\alpha)<L(\beta)$, a contradiction.

\textbf{(3.1)} If $d_{G_{L}}(\alpha,\beta)=1$, then there is a directed edge from $\alpha$ to $\beta$. Thus $i(\alpha,\beta)=0$.

\textbf{(3.2)} Property (2) tells us that $L(\alpha)<L(\beta)$. If
$i(\alpha,\beta)=0$, then $\alpha<\beta$, and hence $d_{G_{L}}(\alpha,\beta)=1$,
a contradiction.

\textbf{(3.3)} Otherwise there would be a directed edge from $\alpha$
to $\beta$, and hence $d_{G_{L}}(\alpha,\beta)=1$, a contradiction.
\end{proof}

Second, the following lemma, while being easy to prove, captures the
essential property we need for the function $L$ and will be used repeatedly in the sequel:
\begin{Lemma}\label{Lemma:tends-to-infty}
Let $L:V(S)\rightarrow\mathbb{R}$ be a function. Then $L$ has finite
sublevel sets if and only if for any infinite sequence $\{\alpha_{n}\}$
of distinct points in $V(S)$ we have
\[
\lim_{n\to\infty}L(\alpha_{n})=+\infty.
\]
\end{Lemma}

It follows that any collection of disjoint simple closed curves can
be extended to a system of decomposing curves by adding longer and
longer simple closed curves:
\begin{Corollary}\label{Corollary:pants-decomposition}
Let $G_{L}$ be the oriented graph induced by a function $L:V(S)\rightarrow\mathbb{R}$
which has finite sublevel sets. Suppose that $m<3g-3$ and $\alpha_{1},\ldots,\alpha_{m}\in V(S)$
satisfy $i(\alpha_{j},\alpha_{k})=0$ for all $j,k$. Then this collection
can be extended to a system of decomposing curves $\alpha_{1},\ldots,\alpha_{3g-3}$
such that
\[
\max_{1\leqslant j\leqslant m}L(\alpha_{j})<L(\alpha_{m+1})<\cdots<L(\alpha_{3g-3}).
\]
\end{Corollary}
\begin{proof}
Let us cut $S$ along $\alpha_{1},\cdots,\alpha_{m}$ and obtain a
(possibly disconnected) surface $S'$ with boundary. Since there are
infinitely many simple closed curves lying on $S'$ and not isotopic
to any component of $\partial S'$, it follows from \prettyref{Lemma:tends-to-infty}
that there exists $\alpha_{m+1}\in V(S)$ such that $i(\alpha_{j},\alpha_{m+1})=0$
for any $j\leqslant m$ and 
\[
\max_{1\leqslant j\leqslant m}L(\alpha_{j})<L(\alpha_{m+1}).
\]

Iterating the above operation we get the desired curves and complete
the proof.
\end{proof}

Third, we characterize non-separating curves in terms of the distance $d_{G_{L}}$:
\begin{Lemma}
\label{Lemma:nonseparating_characterization}Let $L:V(S)\rightarrow\mathbb{R}$
be a function that has finite sublevel sets. Then $\alpha\in V(S)$
is a non-separating curve if and only if the following condition is
satisfied: If $\beta_{1},\beta_{2}\in V(S)$ are distinct vertices
with
\[
d_{G_{L}}(\alpha,\beta_{1})=d_{G_{L}}(\alpha,\beta_{2})=1,
\]
then there exists an infinite sequence $\{\gamma_{k}\}$ of distinct
vertices in $V(S)$ such that for each $k$ we have 
\[
d_{G_{L}}(\alpha,\gamma_{k})=1,\quad d_{G_{L}}(\beta_{1},\gamma_{k})\neq1,\quad d_{G_{L}}(\beta_{2},\gamma_{k})\neq1.
\]
\end{Lemma}

\begin{proof}
There are two implications to be proved.

Let us first assume that $\alpha$ is non-separating. Suppose that
$\beta_{1},\beta_{2}\in V(S)$ are distinct vertices with
\[
d_{G_{L}}(\alpha,\beta_{1})=d_{G_{L}}(\alpha,\beta_{2})=1.
\]
Then property (3.1) of \prettyref{Lemma:distance-basic} tells us
that
\[
i(\alpha,\beta_{1})=i(\alpha,\beta_{2})=0.
\]
Since $\alpha$ is non-separating, there exists a simple closed curve
$\gamma_{0}$ such that
\[
i(\alpha,\gamma_{0})=0,\quad i(\beta_{1},\gamma_{0})>0,\quad i(\beta_{2},\gamma_{0})>0.
\]

Let $\gamma_{k}\coloneqq T_{\beta_{1}}^{k}(\gamma_{0})$. Then $\{\gamma_{k}\}$
is an infinite sequence of distinct vertices in $V(S)$.

Since $i(\alpha,\beta_{1})=i(\alpha,\gamma_{0})=0$, we have $i(\alpha,\gamma_{k})=0$
for any $k\in\mathbb{N}$. Since $i(\beta_{1},\gamma_{0})>0$, we
have $i(\beta_{1},\gamma_{k})>0$ for any $k\in\mathbb{N}$. Since
$i(\beta_{1},\gamma_{0})>0$ and $i(\beta_{2},\gamma_{0})>0$, we
have $i(\beta_{2},\gamma_{k})>0$ for sufficiently large $k$. Without
loss of generality, we may assume that $i(\beta_{2},\gamma_{k})>0$
for any $k\in\mathbb{N}$. To summarize, for any $k\in\mathbb{N}$
we have
\[
i(\alpha,\gamma_{k})=0,\quad i(\beta_{1},\gamma_{k})>0,\quad i(\beta_{2},\gamma_{k})>0.
\]

On the one hand, since $L$ has finite sublevel sets, we have
\[
\lim_{k\to+\infty}L(\gamma_{k})=+\infty.
\]
Hence there exists an integer $N_{\gamma}$ such that $L(\gamma_{k})>L(\alpha)$
for all $k\geqslant N_{\gamma}$. Without loss of generality, we may
assume that $L(\gamma_{k})>L(\alpha)$ for any $k\in\mathbb{N}$.
Combining this with $i(\alpha,\gamma_{k})=0$ we obtain $\alpha<\gamma_{k}$,
and consequently $d_{G_{L}}(\alpha,\gamma_{k})=1$ for any $k\in\mathbb{N}$.

On the other hand, for any $k\in\mathbb{N}$, since $i(\beta_{1},\gamma_{k})>0$
and $i(\beta_{2},\gamma_{k})>0$, property (3.1) of \prettyref{Lemma:distance-basic}
tells us that
\[
d_{G_{L}}(\beta_{1},\gamma_{k})\neq1,\quad d_{G_{L}}(\beta_{2},\gamma_{k})\neq1.
\]

It follows that $\{\gamma_{k}\}$ is the desired sequence, and the
``only if'' implication is proved.

To prove the ``if'' implication, we proceed by contradiction. So suppose
that $\alpha$ is a separating curve.

If we cut $S$ along $\alpha$, we obtain two connected components,
each of which contains infinitely many simple closed curves disjoint
from $\alpha$. Since $L$ has finite sublevel sets, it follows that
there exist $\beta_{1},\beta_{2}\in V(S)$ lying in different components
of $S-\alpha$ such that
\[
\alpha<\beta_{1},\quad\alpha<\beta_{2},\quad i(\beta_{1},\beta_{2})=0,
\]
and hence
\[
d_{G_{L}}(\alpha,\beta_{1})=d_{G_{L}}(\alpha,\beta_{2})=1.
\]
By assumption, there exists an infinite sequence $\{\gamma_{k}\}$
of distinct vertices in $V(S)$ such that for any $k\in\mathbb{N}$
we have 
\[
d_{G_{L}}(\alpha,\gamma_{k})=1,\quad d_{G_{L}}(\beta_{1},\gamma_{k})\neq1,\quad d_{G_{L}}(\beta_{2},\gamma_{k})\neq1.
\]

Since $L$ has finite sublevel sets, we also have
\[
\lim_{n\to\infty}L(\gamma_{n})=+\infty.
\]
Therefore, there exists some $N$ such that
\[
L(\gamma_{N})>\max\{L(\beta_{1}),L(\beta_{2})\}.
\]

On the one hand, since $L(\beta_{1})<L(\gamma_{N})$ and $d_{G_{L}}(\beta_{1},\gamma_{N})\neq1$,
we conclude from properties (3.2) and (3.3) of \prettyref{Lemma:distance-basic}
that $i(\beta_{1},\gamma_{N})>0$. Similarly we have $i(\beta_{2},\gamma_{N})>0$.
On the other hand, $d_{G_{L}}(\alpha,\gamma_{N})=1$ implies that
$\alpha$ and $\gamma_{N}$ are disjoint.

However, $\beta_{1}$ and $\beta_{2}$ lie in different components
of $S-\alpha$, so any curve intersecting both $\beta_{1}$ and $\beta_{2}$
must intersect $\alpha$. This contradicts the fact that $\gamma_{N}$
intersects both $\beta_{1}$ and $\beta_{2}$ while being disjoint
from $\alpha$, and the proof is complete.
\end{proof}

Since the distance $d_{G_{L}}$ is invariant under the action of $\mathrm{Aut}(G_{L})$,
it follows that $\mathrm{Aut}(G_{L})$ preserves the set of non-separating
curves and the set of separating curves:
\begin{corollary}\label{Corollary:cor_pres_separating}
Suppose that $L$ has finite sublevel sets. Let $\varphi\in\mathrm{Aut}(G_{L})$
and $\alpha\in V(S)$. Then $\alpha$ is separating (respectively,
non-separating) if and only if $\varphi(\alpha)$ is separating (respectively,
non-separating).
\end{corollary}

Finally, we are ready to prove our second main theorem.
\begin{proof}[Proof of \prettyref{Theorem:main2}]
Let $\varphi:V(S)\rightarrow V(S)$ be an automorphism of $G_{L}$.
According to \prettyref{Theorem:Ivanov}, it suffices to show that
$\varphi\in\mathrm{Aut}(C(S))$, i.e., for any $\alpha,\beta\in V(S)$
we have
\[
i(\alpha,\beta)=0\Leftrightarrow i(\varphi\cdot\alpha,\varphi\cdot\beta)=0.
\]
Since $\varphi^{-1}$ is also an automorphism of $G_{L}$, it suffices
to show that for any $\alpha,\beta\in V(S)$ with $i(\alpha,\beta)=0$
we have $i(\varphi\cdot\alpha,\varphi\cdot\beta)=0$.

Let $\alpha,\beta\in V(S)$ such that $i(\alpha,\beta)=0$.

If $\alpha<\beta$, then $\varphi\cdot\alpha<\varphi\cdot\beta$,
and hence $i(\varphi\cdot\alpha,\varphi\cdot\beta)=0$. Similarly,
if $\beta<\alpha$, then $i(\varphi\cdot\alpha,\varphi\cdot\beta)=0$.

Therefore, it remains to consider the case when there is no edge between
$\alpha$ and $\beta$ in $G_{L}$. Now $L(\alpha)=L(\beta)$ must hold
since $i(\alpha,\beta)=0$.

Applying \prettyref{Corollary:pants-decomposition} to $\alpha,\beta$
we obtain a system of decomposing curves $\alpha,\beta,\alpha_{2},\cdots,\alpha_{3g-4}$
such that both induced subgraphs $G_{L}[\alpha,\alpha_{2},\cdots,\alpha_{3g-4}]$ and
$G_{L}[\beta,\alpha_{2},\cdots,\alpha_{3g-4}]$ are tournaments, and
\[
\xymatrix{ & \beta\ar[d]\\
\alpha\ar[r] & \alpha_{2}\ar[r] & \cdots\ar[r] & \alpha_{3g-4}.
}
\]

Write
\[
\alpha'=\varphi\cdot\alpha,\quad\beta'=\varphi\cdot\beta,\quad\alpha_{j}'=\varphi\cdot\alpha_{j}.
\]
Then the induced subgraphs $G_{L}[\alpha',\alpha_{2}',\cdots,\alpha_{3g-4}']$ and $G_{L}[\beta',\alpha_{2}',\cdots,\alpha_{3g-4}']$
are tournaments such that
\[
\xymatrix{ & \beta'\ar[d]\\
\alpha'\ar[r] & \alpha_{2}'\ar[r] & \cdots\ar[r] & \alpha_{3g-4}'.
}
\]

\begin{figure}
\begin{centering}
\includegraphics{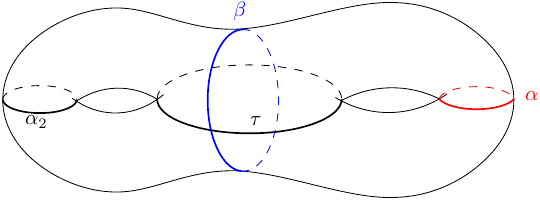}
\par\end{centering}
\caption{\label{fig:za_0q_automorphism-3}The case where $\beta$ is separating.}
\end{figure}

\begin{figure}
\begin{centering}
\includegraphics{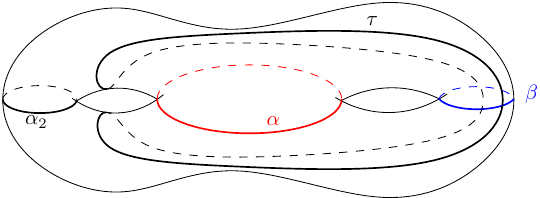}
\par\end{centering}
\caption{\label{fig:za_0q_automorphism-4}The case where $\beta$ is non-separating.}
\end{figure}

As illustrated in \prettyref{fig:za_0q_automorphism-3} and \prettyref{fig:za_0q_automorphism-4},
since $\alpha,\beta,\alpha_{2},\cdots,\alpha_{3g-4}$ form a system
of decomposing curves, there exists $\tau\in V(S)$ disjoint from
$\alpha,\alpha_{2},\cdots,\alpha_{3g-4}$ such that $i(\beta,\tau)>0$.

Let us set $\tau_{k}\coloneqq T_{\beta}^{k}\tau$ and $\tau_{k}'\coloneqq\varphi\cdot\tau_{k}$
for each $k\in\mathbb{N}_{+}$. Then $i(\beta,\tau)>0$ gives $i(\beta,\tau_{k})>0$
for any $k$.

\prettyref{Lemma:tends-to-infty} tells us that
\[
\lim_{k\rightarrow+\infty}L(\tau_{k})=+\infty,
\]
so there exists $N$ such that for any $k\geqslant N$ we have
\[
L(\tau_{k})>\max\{L(\alpha),L(\beta),L(\alpha_{2}),\cdots,L(\alpha_{3g-4})\}.
\]
Replacing $\tau$ with $\tau_{N}$ if necessary, we may assume that
$L(\tau_{k})>\max\{L(\alpha),L(\beta),L(\alpha_{2}),\cdots,L(\alpha_{3g-4})\}$
for any $k$.
It follows that $G_{L}[\alpha,\alpha_{2},\cdots,\alpha_{3g-4},\tau_{k}]$
is a tournament, and hence that $G_{L}[\alpha',\alpha_{2}',\cdots,\alpha_{3g-4}',\tau_{k}']$
is a tournament as well.

Furthermore, since $L(\tau_{k})>\max\{L(\alpha),L(\alpha_{2}),\cdots,L(\alpha_{3g-4})\}$, we have
\[
\xymatrix{ & \beta\ar[d]\\
\alpha\ar[r] & \alpha_{2}\ar[r] & \cdots\ar[r] & \alpha_{3g-4}\ar[r] & \tau_{k},
}
\]
which gives
\[
\xymatrix{ & \beta'\ar[d]\\
\alpha'\ar[r] & \alpha_{2}'\ar[r] & \cdots\ar[r] & \alpha_{3g-4}'\ar[r] & \tau_{k}',
}
\]
and hence $L(\beta')<L(\tau_{k}')$ for any $k$.

We claim that $i(\beta',\tau_{k}')>0$. Otherwise, we would have $\beta'<\tau_{k}'$,
which implies $\beta<\tau_{k}$ and contradicts the fact that $i(\beta,\tau_{k})>0$.

\textbf{For the rest of the argument, we proceed by contradiction and assume
that $i(\alpha',\beta')>0$.}

Recall that we say that the type of a surface is $(n_{1},n_{2})$ if it
is a surface of genus $n_{1}$ with $n_{2}$ boundary components.

Cutting $S$ along $\alpha_{2}',\cdots,\alpha_{3g-4}'$ we obtain
some bordered subsurfaces, among which there is a unique subsurface
$S_{1}$ that is not a surface of type $(0,3)$. Indeed, $S_{1}$
must be a surface of type $(0,5)$ or $(1,2)$.
Since both $G_{L}[\alpha',\alpha_{2}',\cdots,\alpha_{3g-4}',\tau_{k}']$
and $G_{L}[\beta',\alpha_{2}',\cdots,\alpha_{3g-4}']$ are
tournaments, we conclude that $\alpha',\beta',\tau_{k}'$ must lie
in the interior of $S_{1}$.

The proof splits into two cases, depending on whether $\alpha'$ or
$\beta'$ separates $S_{1}$ into two subsurfaces.

\textbf{First case: $\alpha'$ or $\beta'$ separates $S_{1}$ into
two subsurfaces.}

Without loss of generality, we may assume that $\alpha'$ separates
$S_{1}$ into two subsurfaces, exactly one of which has type $(0,3)$.
We denote by $S_{1}''$ the subsurface of type $(0,3)$, and denote
by $S_{1}'$ the other one.

Since $G_{L}[\alpha',\alpha_{2}',\cdots,\alpha_{3g-4}',\tau_{k}']$
is a tournament, it follows that $\tau_{k}'$ must lie in $S_{1}'$.
 
\begin{figure}
\begin{centering}
\includegraphics{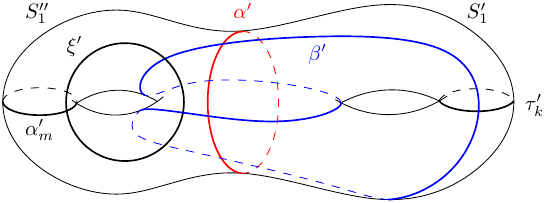}
\par\end{centering}
\caption{\label{fig:za_0q_automorphism-7}The case where $S_{1}$ has type
$(1,2)$ and $\alpha'$ separates $S_{1}$ into two subsurfaces.}
\end{figure}

As illustrated in \prettyref{fig:za_0q_automorphism-7}, there exist
some $\alpha_{m}'$ and $\xi'$ such that for any $j\in\mathbb{N}_{+}$
the subgraph induced in $C(S)$ by $\{\alpha',\beta',\alpha_{m}',\tau_{j}',\xi'\}$
is
\[
\xymatrix{ & \alpha'\ar@{-}[r]\ar@{-}[d]\ar@{-}[dr] & \xi'\\
\beta'\ar@{-}[r] & \alpha_{m}'\ar@{-}[r] & \tau_{j}'\ar@{-}[u]
}
.
\]

We set $\xi_{k}'\coloneqq T_{\alpha_{m}'}^{k}\xi'$ and $\xi_{k}\coloneqq\varphi^{-1}\cdot\xi_{k}'$
for each $k\in\mathbb{N}_{+}$. Since $L$ has finite sublevel sets,
\prettyref{Lemma:tends-to-infty} tells us that 
\[
\lim_{k\rightarrow+\infty}L(\xi_{k})=+\infty,
\]
hence there exists $N_{\xi}$ such that $L(\beta)<L(\xi_{k})$
for any $k\geqslant N_{\xi}$.

Fix $k_{0}\geqslant N_{\xi}$. Then \prettyref{Lemma:tends-to-infty}
tells us that 
\[
\lim_{j\rightarrow+\infty}L(\tau_{j}')=+\infty,
\]
hence there exists $N_{k_{0}}$ such that $L(\xi_{k_{0}}')<L(\tau_{j}')$
for any $j\geqslant N_{k_{0}}$. Observe that $i(\xi_{k_{0}}',\tau_{j}')\equiv0$,
and therefore for any $j\geqslant N_{k_{0}}$ we have $\xi_{k_{0}}'<\tau_{j}'$,
which gives $\xi_{k_{0}}<\tau_{j}$, and hence $i(\xi_{k_{0}},\tau_{j})=0$.
Combining this with
\[
\lim_{j\rightarrow+\infty}\frac{\tau_{j}}{j\cdot i(\tau,\beta)}=\beta\quad\text{in}\quad\mathcal{MF},
\]
we obtain 
\[
i(\xi_{k_{0}},\beta)=\lim_{j\rightarrow+\infty}i\left(\xi_{k_{0}},\frac{\tau_{j}}{j\cdot i(\tau,\beta)}\right)=\lim_{j\rightarrow+\infty}\frac{i(\xi_{k_{0}},\tau_{j})}{j\cdot i(\tau,\beta)}=0.
\]
It follows from $L(\beta)<L(\xi_{k_{0}})$ that $\beta<\xi_{k_{0}}$,
which gives $\beta'<\xi_{k_{0}}'$, and hence $i(\beta',\xi_{k_{0}}')=0$.

However, we have $i(\alpha_{m}',\beta')=0$ and $i(\xi',\beta')>0$.
It follows that 
\[
i(\beta',\xi_{k_{0}}')=i\left(T_{\alpha_{m}'}^{k_{0}}\beta',T_{\alpha_{m}'}^{k_{0}}\xi'\right)=i(\beta',\xi')>0,
\]
a contradiction!

\textbf{Second case: Neither $\alpha^{\prime}$ nor $\beta^{\prime}$
separates $S_{1}$ into two subsurfaces.}

In this case, neither $\alpha'$ nor $\beta'$ is a separating curve
in $S$, so it follows from \prettyref{Corollary:cor_pres_separating}
that neither $\alpha$ nor $\beta$ is a separating curve in $S$.

\begin{figure}[h]
\begin{centering}
\includegraphics{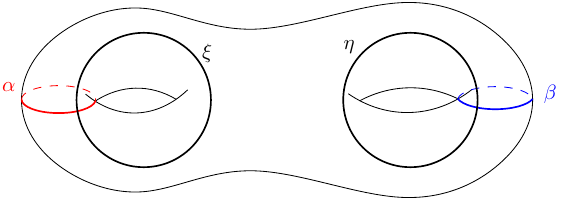}
\par\end{centering}
\caption{\label{fig:za_0q_ns-ns}}
\end{figure}

As illustrated in \prettyref{fig:za_0q_ns-ns}, we can choose distinct
curves $\xi,\eta\in V(S)$ such that the subgraph induced in $C(S)$
by $\{\alpha,\beta,\xi,\eta\}$ is
\begin{equation}
\xymatrix{\alpha\ar@{-}[r]\ar@{-}[d] & \beta\ar@{-}[d]\\
\eta\ar@{-}[r] & \xi
}
.\label{eq:main2-1}
\end{equation}

For each $k\in\mathbb{N}_{+}$ we set
\begin{align*}
\xi_{k} & \coloneqq T_{\alpha}^{k}\xi, & \eta_{k} & \coloneqq T_{\beta}^{k}\eta,\\
\xi_{k}' & \coloneqq\varphi\cdot\xi_{k}, & \eta_{k}' & \coloneqq\varphi\cdot\eta_{k}.
\end{align*}
It follows from \prettyref{eq:main2-1} that the subgraph induced
in $C(S)$ by $\{\alpha,\beta,\xi_{j},\eta_{k}\}$ is
\begin{equation}
\xymatrix{\alpha\ar@{-}[r]\ar@{-}[d] & \beta\ar@{-}[d]\\
\eta_{k}\ar@{-}[r] & \xi_{j}
}
.\label{eq:main2-2}
\end{equation}
\prettyref{Lemma:tends-to-infty} tells us that
\begin{equation}
\lim_{k\rightarrow+\infty}L(\xi_{k})=\lim_{k\rightarrow+\infty}L(\xi_{k}')=\lim_{k\rightarrow+\infty}L(\eta_{k})=\lim_{k\rightarrow+\infty}L(\eta_{k}')=+\infty.\label{eq:main2-3}
\end{equation}
Thus there exists $N_{0}\in\mathbb{N}_{+}$ such that for any $k\geqslant N_{0}$
we have $L(\beta)<L(\xi_{k})$ and $L(\alpha')<L(\xi_{k}')$.

Let us set $\sigma=\xi_{N_{0}}$ and $\sigma'=\varphi\cdot\sigma$.
Then $L(\beta)<L(\sigma)$ and $L(\alpha')<L(\sigma')$. Moreover,
it follows from \prettyref{eq:main2-2} that the subgraph induced
in $C(S)$ by $\{\alpha,\beta,\sigma,\eta_{k}\}$ is
\begin{equation}
\xymatrix{\alpha\ar@{-}[r]\ar@{-}[d] & \beta\ar@{-}[d]\\
\eta_{k}\ar@{-}[r] & \sigma
}
.\label{eq:main2-4}
\end{equation}
Combining this with $L(\beta)<L(\sigma)$ we get $\beta<\sigma$,
which yields $\beta'<\sigma'$, and hence $i(\beta',\sigma')=0$.
Furthermore, from $(\alpha,\sigma)\notin G_{L}$ we obtain $(\alpha',\sigma')\notin G_{L}$,
hence $L(\alpha')<L(\sigma')$ implies $i(\alpha',\sigma')\neq0$.
To summarize, the subgraph induced in $C(S)$ by $\{\alpha',\beta',\sigma'\}$
is
\begin{equation}
\xymatrix{\alpha' & \beta'\ar@{-}[d]\\
 & \sigma'
}
.\label{eq:main2-7}
\end{equation}
\prettyref{Lemma:tends-to-infty} tells us that
\[
\lim_{k\rightarrow+\infty}L(\varphi^{-1}\cdot T_{\alpha'}^{k}\sigma')=+\infty.
\]
Combining this with $i(\alpha',\beta')>0$ we conclude that there
exists $k_{0}\in\mathbb{N}_{+}$ such that $L(\beta)<L(\varphi^{-1}\cdot T_{\alpha'}^{k_{0}}\sigma')$
and $i(T_{\alpha'}^{k_{0}}\sigma',\beta')>0$. Let us set $\gamma'=T_{\alpha'}^{k_{0}}\sigma'$
and $\gamma=\varphi^{-1}\cdot\gamma'$. Then we have
\begin{equation}
L(\beta)<L(\gamma)\quad\text{and}\quad i(\gamma',\beta')>0.\label{eq:main2-11}
\end{equation}
It follows from \prettyref{eq:main2-3} that there exists $N_{1}\in\mathbb{N}_{+}$
such that for any $k\geqslant N_{1}$ we have 
\begin{equation}
L(\eta_{k})>\max\{L(\alpha),L(\sigma),L(\gamma)\}\quad\text{and}\quad L(\eta_{k}')>L(\gamma').\label{eq:main2-5}
\end{equation}
Replacing $\eta$ with $\eta_{N_{1}}$ if necessary, we may assume that
\prettyref{eq:main2-5} holds for every $k\in\mathbb{N}_{+}$. Combining
\prettyref{eq:main2-4} and \prettyref{eq:main2-5} we see that the
subgraph induced in $G_{L}$ by $\{\alpha,\beta,\sigma,\eta_{k}\}$
is
\[
\xymatrix{\alpha\ar[d] & \beta\ar[d]\\
\eta_{k} & \sigma\ar[l]
}
,
\]
and hence the subgraph induced in $G_{L}$ by $\{\alpha',\beta',\sigma',\eta_{k}'\}$
is
\begin{equation}
\xymatrix{\alpha'\ar[d] & \beta'\ar[d]\\
\eta_{k}' & \sigma'\ar[l]
}
.\label{eq:main2-8}
\end{equation}
Since $\gamma'=T_{\alpha'}^{k_{0}}\sigma'$, combining \prettyref{eq:main2-8}
with \prettyref{eq:main2-5} we conclude that the subgraph induced
in $G_{L}$ by $\{\alpha',\beta',\eta_{k}',\gamma'\}$ is
\[
\xymatrix{\alpha'\ar[d] & \beta'\\
\eta_{k}' & \gamma'\ar[l]
}
,
\]
and hence the subgraph induced in $G_{L}$ by $\{\alpha,\beta,\eta_{k},\gamma\}$
is
\begin{equation}
\xymatrix{\alpha\ar[d] & \beta\\
\eta_{k} & \gamma\ar[l]
}
,\label{eq:main2-10}
\end{equation}
which gives $i(\eta_{k},\gamma)\equiv0$.

However, combining $i(\eta_{k},\gamma)\equiv0$ with 
\[
\lim_{k\rightarrow+\infty}\frac{\eta_{k}}{k\cdot i(\eta,\beta)}=\beta\quad\text{in}\quad\mathcal{MF},
\]
we obtain
\[
i(\beta,\gamma)=\lim_{k\rightarrow+\infty}i\left(\frac{\eta_{k}}{k\cdot i(\eta,\beta)},\gamma\right)=\lim_{k\rightarrow+\infty}\frac{i(\eta_{k},\gamma)}{k\cdot i(\eta,\beta)}=0,
\]
and hence \prettyref{eq:main2-11} implies $\beta<\gamma$, contrary
to \prettyref{eq:main2-10}.

In conclusion, assuming $i(\alpha',\beta')=0$ necessarily leads to
a contradiction, and therefore the proof is complete.
\end{proof}

\subsection{Periodicity}

We say that $A:E^{c}\rightarrow\mathbb{R}_{+}$ is $\mathrm{Mod}^{\pm}(S)$-\emph{invariant}
if for any $(\alpha,\beta)\in E^{c}$ and $h\in\mathrm{Mod}^{\pm}(S)$
we have $A(h(\alpha),h(\beta))=A(\alpha,\beta)$.
\begin{Lemma}\label{Lemma:periodic}
Let $L:V(S)\rightarrow\mathbb{R}$ be a Dehn quasi-homothetic function
of type $(f,A)$, where $A$ is $\mathrm{Mod}^{\pm}(S)$-invariant. Suppose
that $L$ is positive and has finite sublevel sets. Then for any $\varphi\in\mathrm{Aut}(G_{L})$
and $\alpha\in V(S)$ we have:

\emph{(i)} $L(\alpha)=L(\varphi\cdot\alpha)$,

\emph{(ii)} $\#\{\varphi^{k}\cdot\alpha:k\in\mathbb{Z}\}<\infty$,

\emph{(iii)} $\varphi$ is periodic.
\end{Lemma}
\begin{proof}
Let us set $L_{1}=L$ and $L_{2}=L\circ\varphi$.

Since $L$ has finite sublevel sets, \prettyref{Theorem:main2} tells
us that $\varphi$ is induced by some $h\in\mathrm{Mod}^{\pm}(S)$.
Because $L$ has type $(f,A)$, for any $(\alpha,\beta)\in E^{c}(S)$
we have
\[
\lim_{n\rightarrow+\infty}\frac{L(T_{\alpha}^{n}\beta)}{f(n)}=A(\alpha,\beta)\cdot L(\alpha),
\]
and hence for $(h(\alpha),h(\beta))\in E^{c}(S)$ it holds that
\[
\lim_{n\rightarrow+\infty}\frac{L(T_{h(\alpha)}^{n}h(\beta))}{f(n)}=A(h(\alpha),h(\beta))\cdot L(h(\alpha)),
\]
which gives
\begin{align*}
\lim_{n\rightarrow+\infty}\frac{L_{2}(T_{\alpha}^{n}\beta)}{f(n)} & =\lim_{n\rightarrow+\infty}\frac{L(h(T_{\alpha}^{n}\beta))}{f(n)}\\
 & =\lim_{n\rightarrow+\infty}\frac{L(T_{h(\alpha)}^{n}h(\beta))}{f(n)}\\
 & =A(h(\alpha),h(\beta))\cdot L(h(\alpha))\\
 & =A(\alpha,\beta)\cdot L_{2}(\alpha).
\end{align*}
It follows that $L_{2}:V(S)\rightarrow\mathbb{R}$ is a Dehn quasi-homothetic
function of type $(f,A)$ as well.

Since $\varphi$ is an automorphism of $G_{L}$, for any $\alpha,\beta\in V(S)$
we have
\[
\alpha<\beta\Leftrightarrow\varphi\cdot\alpha<\varphi\cdot\beta,
\]
and hence $\alpha<\beta$ in $G_{L_{1}}$ if and only if $\alpha<\beta$
in $G_{L_{2}}$. Thus $L_{1}$ and $L_{2}$ induce the same oriented
graph.

As $L_{1}$ and $L_{2}$ have the same type and induce the same oriented
graph, we conclude from \prettyref{Theorem:main1} that there exists
$c>0$ such that $L_{1}=c\cdot L_{2}$, i.e., $L=c\cdot L\circ\varphi$.
Consequently, for any $\alpha\in V(S)$ we have
\[
\cdots=c^{-k}\cdot L(\varphi^{-k}\cdot\alpha)=\cdots=L(\alpha)=c\cdot L(\varphi\cdot\alpha)=\cdots=c^{k}\cdot L(\varphi^{k}\cdot\alpha)=\cdots,
\]
which implies $c=1$ since $L$ has finite sublevel sets, and (i)
follows.

Because $L$ has finite sublevel sets, it follows from
\[
\cdots=L(\varphi^{-k}\cdot\alpha)=\cdots=L(\alpha)=L(\varphi\cdot\alpha)=\cdots=L(\varphi^{k}\cdot\alpha)=\cdots
\]
that $\#\{\varphi^{k}\cdot\alpha:k\in\mathbb{Z}\}<\infty$, and (ii)
is proved. Furthermore, since $\alpha$ is arbitrary, we conclude
from Theorem 2.7 of \cite{casson1988automorphisms} that $\varphi$
is periodic, as claimed in (iii).
\end{proof}

\section{Applications to length functions}\label{sec:applications}

\subsection{Proof of \prettyref{Corollary:Cor3}}
Assume that $C_{X}(S)=C_{Y}(S)$.

The proof is divided into three cases.

\textbf{First case:} Both $C_{X}(S)$ and $C_{Y}(S)$ are induced
by the hyperbolic length functions.

Recall that we denote the hyperbolic length functions of $X$ and
$Y$ by $L_{X}$ and $L_{Y}$, respectively. In \prettyref{sub:examples}
we showed that the hyperbolic length functions are Dehn quasi-homothetic
functions of type
\[
(k\mapsto k,(\alpha,\beta)\mapsto i(\alpha,\beta)).
\]
Since $L_{X}$ and $L_{Y}$ have the same type and induce the same
oriented graph, we conclude from \prettyref{Theorem:main1} that there
exists $k>0$ such that $L_{X}=k\cdot L_{Y}$, which implies
\[
\ln\sup_{\alpha\in V(S)}\frac{L_{X}(\alpha)}{L_{Y}(\alpha)}=\ln k.
\]
Combining this with \prettyref{eq:Thurston-distance} we get $d_{\mathrm{Th}}(X,Y)=-\ln k$
and $d_{\mathrm{Th}}(Y,X)=\ln k$, and hence $k=1$ since $d_{\mathrm{Th}}\geqslant0$.
Thus $d_{\mathrm{Th}}(X,Y)=0$, which gives $X=Y$, as desired.

\textbf{Second case:} Both $C_{X}(S)$ and $C_{Y}(S)$ are induced
by the extremal length functions.

Recall that we denote the extremal length functions of $X$ and $Y$
by $\mathrm{Ext}_{X}$ and $\mathrm{Ext}_{Y}$, respectively. In \prettyref{sub:examples}
we showed that the extremal length functions are Dehn quasi-homothetic
functions of type
\[
(k\mapsto k,(\alpha,\beta)\mapsto i(\alpha,\beta)).
\]
Since $\mathrm{Ext}_{X}$ and $\mathrm{Ext}_{Y}$ have the same type
and induce the same oriented graph, we conclude from \prettyref{Theorem:main1}
that there exists $k>0$ such that $\mathrm{Ext}_{X}=k\cdot\mathrm{Ext}_{Y}$,
which implies
\[
\frac{1}{2}\ln\sup_{\alpha\in V(S)}\frac{\mathrm{Ext}_{X}(\alpha)}{\mathrm{Ext}_{Y}(\alpha)}=\frac{1}{2}\ln k.
\]
Hence it follows from \prettyref{Theorem:Teich-distance} that $\ln k=-\ln k$,
which yields that the Teichmüller distance between $X$ and $Y$ is
zero, and so $X=Y$, as desired.

\textbf{Third case:} Exactly one of $C_{X}(S)$ and $C_{Y}(S)$ is
induced by the hyperbolic length function.

Without loss of generality, we may assume that $C_{X}(S)$ is induced
by the hyperbolic length function $L_{X}$, and $C_{Y}(S)$ is induced
by the extremal length function $\mathrm{Ext}_{Y}$. Since $L_{X}$
and $\sqrt{\mathrm{Ext}_{Y}}$ have the same type and induce the same
oriented graph, we conclude from \prettyref{Theorem:main1} that there
exists $k>0$ such that $L_{X}=k\cdot\sqrt{\mathrm{Ext}_{Y}}$. Then
for any $\alpha\in V(S)$ we have
\[
\mathrm{Ext}_{Y}(\alpha)=\frac{L_{X}^{2}(\alpha)}{k^{2}},
\]
which contradicts \prettyref{Lemma:hyp-neq-ext}. Therefore, this
case is impossible, and the proof is complete.

\subsection{Proof of \prettyref{Corollary:Cor6}}
Let $\varphi$ be an automorphism of $C_{X}(S)$.

The proof is divided into two cases.

\textbf{First case:} $C_{X}(S)$ is induced by the hyperbolic length function
$L_{X}$.

In \prettyref{sub:examples} we showed that the hyperbolic length
functions are Dehn quasi-homothetic functions of type
\[
(k\mapsto k,(\alpha,\beta)\mapsto i(\alpha,\beta)).
\]
Observe that for any $\alpha,\beta\in V(S)$ and $h\in\mathrm{Mod}^{\pm}(S)$
we have $i(h(\alpha),h(\beta))=i(\alpha,\beta)$. It follows from
\prettyref{Lemma:periodic} that for any $\alpha\in V(S)$ we have
$L_{X}(\alpha)=L_{X}(\varphi\cdot\alpha)$.

On the other hand, \prettyref{Theorem:main2} tells us that $\varphi$
is induced by a self-homeomorphism $h$ of $S$. Furthermore, we may
assume that $h:S\rightarrow S$ is a diffeomorphism. Let
$Y=h^{*}X\in\mathcal{T}(S)$. Then for any $\alpha\in V(S)$, we have
\[
L_{Y}(\alpha)=L_{X}(h(\alpha))=L_{X}(\varphi\cdot\alpha)=L_{X}(\alpha).
\]
Hence it follows from the $9g-9$ theorem that $X=Y$. Finally, by
the definition of Teichmüller space, $h$ is isotopic to
an isometry $H$ of $X$, and hence $\varphi$ is induced by $H$.

\textbf{Second case:} $C_{X}(S)$ is induced by the extremal length
function $\mathrm{Ext}_{X}$.

Similarly to the first case, we can show that for any $\alpha\in V(S)$,
it holds that
\[
\mathrm{Ext}_{Y}(\alpha)=\mathrm{Ext}_{X}(\alpha).
\]
Hence it follows from \prettyref{Theorem:Teich-distance} that $X=Y$,
and so $\varphi$ is induced by an isometry of $X$, which completes the proof.

\section{Conclusion}

The most interesting Dehn quasi-homothetic functions on $C(S)$ are
those of type
\[
(k\mapsto k,(\alpha,\beta)\mapsto i(\alpha,\beta)),
\]
which may be called Dehn quasi-linear functions. Let 
$\mathcal{DQL}(S)$ donote the space of Dehn quasi-linear functions.
Then $\mathcal{DQL}(S)$ contains:

(1) length functions induced by metrics with non-positive curvature,
such as hyperbolic metrics and flat metrics;

(2) the square root of an extremal length function;

(3) the geometric intersection number with a given measured foliation or lamination.

Equipping $\mathcal{DQL}(S)$ with Thurston metric, we obtain a natural
isometric embedding of the Teichmüller space $\mathcal{T}(S)$ equipped with Thurston
metric.

Therefore, $\mathcal{DQL}(S)$ provides a unified framework to study
these geometric objects. Moreover, although $\mathcal{DQL}(S)$ looks
smaller than the space of geodesic currents, it contains the square roots of extremal
length functions, while the square root of any extremal length function cannot be
realized as the geometric intersection number with a geodesic current.

\prettyref{Theorem:main2} suggests that the oriented graph induced
by a function often retains a lot of information about $C(S)$. Therefore,
the automorphism groups of such oriented graphs may offer an alternative
perspective for studying subgroups of mapping class groups.

In another paper that we are preparing, we will show that the oriented
graph induced by a measured foliation or lamination, via the geometric
intersection number, has only automorphisms that is induced by self-homeomorphisms
of $S$. The style of our proof is similar to that of \prettyref{Lemma:nonseparating_characterization},
but far more complicated.

\printbibliography

\end{document}